\documentclass[final]{siamart1116}

\usepackage{etex}
\usepackage{amssymb,graphicx,pst-all,color,enumerate}
\usepackage[latin1]{inputenc}
\usepackage{xspace}
\usepackage{algorithm}
\usepackage{algorithmic}
\usepackage{tikz}

\newsiamremark{assumption}{Assumption}
\newsiamremark{remark}{Remark}
\newsiamthm{procedure}{Procedure}
\numberwithin{theorem}{section}

\def\norm#1{\left\|#1\right\|}
\renewcommand{\sp}[2]{\langle#1,#2\rangle}
\newcommand{\VEC}[2]{\left(\begin{array}{c} #1 \\ #2 \end{array}\right)}
\newcommand{\MATRIX}[4]{\left(\begin{array}{rr} #1 & #2 \\ #3 & #4 \end{array}\right)}
\newcommand{\R}{\mathbb R}

\newcommand{\N}{\mathbb N}
\newcommand{\A}{\mathbb A}
\newcommand{\mcP}{\mathcal{P}}
\newcommand{\mcV}{\mathcal{V}}
\newcommand{\curl}{\mbox{curl}}
\newcommand{\oO}{\overline{\Omega}}
\newcommand{\Ot}{\Omega_c}
\newcommand{\Ott}{\Omega_I}
\newcommand{\sigs}{\sigma^\star}
\newcommand{\Ca}{C^{0,\alpha}}
\newcommand{\Conea}{C^{1,\alpha}}
\newcommand{\cc}{\subset\subset}
\newcommand{\bn}{\mathbf{n}}
\newcommand{\MR}{\mathbf{r}}
\newcommand{\Bz}{Harmonic $B_z$ Algorithm\xspace}
\newcommand{\RBz}{Reduced Basis Harmonic $B_z$ Algorithm\xspace}
\newcommand{\Amplaw}{Amp\`{e}re's law\xspace}

\title{Magnet Resonance Electrical Impedance Tomography (MREIT): Convergence and Reduced Basis Approach\thanks{Submitted to the editors \today}}

\author{
  Dominik Garmatter\thanks{Department of Mathematics, Goethe University Frankfurt, Germany (\email{garmatter@math.uni-frankfurt.de}, \email{harrach@math.uni-frankfurt.de})}
  \and
  Bastian Harrach\footnotemark[2] 
}

\headers{MREIT: convergence and Reduced Basis approach}{D. Garmatter and B. Harrach}

\begin{document}

\maketitle

\begin{abstract}
This article considers the inverse problem of Magnet resonance electrical impedance tomography (MREIT) in two dimensions. A rigorous mathematical framework for this inverse problem as well as the existing \Bz as a solution algorithm are presented. The convergence theory of this algorithm is extended, such that the usage an approximative forward solution of the underlying partial differential equation (PDE) in the algorithm is sufficient for convergence. Motivated by this result, a novel algorithm is developed where it is the aim to speed-up the existing \Bz. This is achieved by combining it with an adaptive variant of the reduced basis method, a model order reduction technique. In a numerical experiment a high-resolution image of the shepp-logan phantom is reconstructed and both algorithms are compared.
\end{abstract}



\section{Introduction}
\label{sec:introduction_MREITpaper}

Magnet resonance electrical impedance tomography (MREIT) is an imaging modality developed over the course of the last three decades. In order to obtain data, surface electrodes are attached onto the imaging subject, e.g. the human body, whilst the object resides inside an MRI scanner. Injecting current through the electrodes then results in a change of the magnetic flux density $\textbf{B} = (B_x,B_y,B_z)$ inside the subject and the MRI scanner can detect this change in the magnetic field. The aim of the method is the determination of the electrical conductivity $\sigma$ of the imaging subject from this measured data. This paper deals with the the $B_z$-based MREIT approach, that is feasible in practice: it is assumed that only $B_z$ is available where the $z$-direction is the direction of the main magnetic field of the MRI scanner (earlier techniques utilized the whole magnetic field \textbf{B}, but cumbersome subject rotations are then necessary to acquire all three components). 
In order to solve the inverse problem of determining $\sigma$ from $B_z$, one can apply the well-known \emph{\Bz} which was proposed by Seo et al. \cite{HarmonicBz_mainpaper} and has since then been extensively studied, see, e.g., \cite{HarmonicBz_2003b, Kwon_Seo_Woo2006_2dMREIT_uniq, seo_woo_2008_MREIT, Jijun_Seo_Woo_MREITconv, Jijun_Seo_Woo_MREITerrest, seo_woo_2010_MREITcompilation} and the references therein.

The historical motivation for the development of MREIT techniques is the Electrical impedance tomography (EIT), see, e.g., \cite{cheney_isaacson_newell_1999EIT, newell1988electric, adler2011electrical, lionheart2004eit, uhlmann2009electrical} for a broad overview. EIT is known to be severely ill-posed and nonlinear such that the spatial resolution of a reconstruction is (usually) poor (on the other hand EIT shines with an excellent temporal resolution, cf. time-difference EIT in the above cited works). Consequently, improving the spatial resolution of conductivity images was the driving force for the development of MREIT techniques.

Before the contributions of this paper to the field are described, the basic setting and the key identity of the \Bz are recapitulated. Let the imaging subject reside in a bounded domain $\Omega\subset\R^3$ with two pairs of surface electrodes attached to it and $E_1^\pm,\, E_2^\pm$ denote the respective parts of $\partial\Omega$ where the electrodes are attached. Furthermore, let $\sigma\in C^1(\oO)$. Each of the two input currents $I_1,\,I_2$ (one per electrode pair) induces a magnetic field and the respective z-components $B_z^1,\, B_z^2$ are the observable data. The physical motivation for $B_z$-based MREIT is the implicit connection between the unknown conductivity $\sigma$ and the observable data $B_z^1,\, B_z^2$ via the \emph{Biot-Savart law:} for $j=1,2$, where $j$ specifies the active electrode pair throughout this article, and $\MR = (x,y,z)\in\Omega$ it is
\begin{align*}
B_z^j(\MR) = \frac{\mu_0}{4\pi}\int_\Omega\sigma(\MR)\frac{(x-x')\frac{\partial u_j^\sigma}{\partial y}(\MR') - (y-y')\frac{\partial u_j^\sigma}{\partial x}(\MR')}{\vert \MR-\MR'\vert^3}\mathrm{d}\MR',
\end{align*}
with $\mu_0$ the magnetic constant of the free space. $u_j^\sigma$ denotes the electrical potential that satisfies the \emph{shunt model}, i.e. 
\begin{align*}
\nabla\cdot(\sigma\nabla u_j^\sigma) = 0\quad\mbox{in}\quad \Omega\\
I_j = \int_{E_j^+}\sigma\frac{\partial u_j^\sigma}{\partial \mathbf{n}} \mathrm{d}s = - \int_{E_j^-} \sigma\frac{\partial u_j^\sigma}{\partial\mathbf{n}} \mathrm{d}s,\quad 
\nabla u_j^\sigma \times \mathbf{n} = 0,\quad\mbox{on}\quad E_j^+\cup E_j^-,\\
\sigma\frac{\partial u_j^\sigma}{\partial\mathbf{n}} = 0\quad\mbox{on}\quad\partial\Omega\setminus\overline{\left(E_j^+\cup E_j^-\right)},
\end{align*}
where $\mathbf{n}$ denotes the outward unit normal vector and $\times$ denotes the cross product.
The \Bz is then an iteration based upon the following identity (here in the logarithmic formulation), which is obtained by applying the $\curl$-operator on both sides of \Amplaw:
\begin{align}
\label{eq:maxwell_relation_MREITpaper}
\nabla_{xy} \ln\sigma = \frac{1}{\mu_0}(\sigma\A[\sigma])^{-1}\VEC{\nabla^2 B_z^1}{\nabla^2 B_z^2},\quad\mbox{with}\quad\A[\sigma] = \MATRIX{\frac{\partial u_1}{\partial y}}{-\frac{\partial u_1}{\partial x}}{\frac{\partial u_2}{\partial y}}{-\frac{\partial u_2}{\partial x}},
\end{align}
where $\nabla^2$ always denotes the Laplace operator throughout this article and $\nabla_{xy}$ is the gradient in $x$ and $y$ direction.

For locally cylindrical subjects with a conductivity that is hardly changing alongside the $z$-direction, the corresponding MREIT problem can entirely be formulated in two space dimensions, see, e.g., \cite{Jijun_Seo_Woo_MREITconv, Jijun_Seo_Woo_MREITerrest}. This paper will consider this two-dimensional MREIT problem, which is feasible in practice for the limbs and the thorax of the human body.

The contribution of this paper to the field is the following: although there have been many advanced numerical studies in MREIT, the convergence analysis did so far only consider the idealized case in which the exact forward solution $u_j^\sigma$ is available for the \Bz. Of course, this is not the case in a numerical study (where for instance only a finite element approximation is available), such that numerical convergence of the algorithm remains an open question.
This paper provides a rigorous and complete mathematical framework as well as a convergence result for the inverse problem in question. The convergence result is based on and at the same time an extension of the existing convergence theory, see \cite{Jijun_Seo_Woo_MREITconv, Jijun_Seo_Woo_MREITerrest}, such that the usage of an approximation of $u_j^\sigma$ in the \Bz is sufficient for convergence. As a consequence, actual numerical convergence of the algorithm is achieved. Furthermore, the potential to use an approximative forward solution instead of the exact one opens up the possibility of combining the existing \Bz with model order reduction techniques in order to develop novel algorithms that retain the accuracy in the reconstruction but are computationally faster. 
The reduced basis method, see, e.g., \cite{RB_master_paper, Ha14RBTut} for a general survey, as a model order reduction technique is presented, where the main task of the method is the construction of a low-dimensional reduced basis space, e.g. via snapshots that are forward solutions for relevant parameters, followed by Galerkin projection onto this space.
This novel algorithm will utilize an \emph{adaptive} reduced basis approach: new parameter values for the enrichment of the reduced basis space are found by projecting the inversion algorithm on it and iterating this projected algorithm. By alternatively updating the reduced basis space and reprojecting  the inversion algorithm onto it, the solution of the inverse problem and the construction of the reduced basis space are achieved simultaneously. 
This adaptive approach was outlined for the nonlinear Landweber method in \cite{G_Harrach_Haasdonk} and is based on ideas developed in \cite{Druskin_Zaslavski, cuiWilcox2014datadriveninversion, lass2014PHDThesis, zahr2014progressive}.

The remainder of this paper is organizes as follows. In section \ref{sec:MREIT_MREITpaper} the  forward and inverse problem in question are presented. Required results for the convergence theorem are derived and the theorem itself is proven. Section \ref{sec:RBM_and_MREIT_MREITpaper} contains a short presentation of the reduced basis method as well as the development of the novel algorithm including numerical results. Conclusions are made in section \ref{sec:conclusion_MREITpaper}.

\section{Magnet Resonance Electrical Impedance Tomography (MREIT)}
\label{sec:MREIT_MREITpaper}

This section provides the mathematical framework, i.e. the forward and the inverse problem of MREIT, the solution algorithms for the inverse problem and the convergence theorem including various minor results. As mentioned in the introduction the focus of this paper is a convergence analysis for the two-dimensional MREIT problem (and in section \ref{sec:RBM_and_MREIT_MREITpaper} the speed-up of the \Bz), such that the mathematical setting will be chosen accordingly. We refer to \cite{seo_woo_2010_MREITcompilation, HarmonicBz_mainpaper, Jijun_Seo_Woo_MREITconv} for a detailed motivation as well as an overview of the MREIT problem.

\subsection{Problem formulation}

For the remainder of this article, let $\Ot\cc\Ott\cc\Omega\subset\R^2$ be $\Conea$-domains with $\alpha\in (0,1)$, $E_1^\pm,\, E_2^\pm$ denote the respective parts of $\partial\Omega$ where the electrodes are attached and $I_1,\,I_2$ be the input currents corresponding to the electrodes, see figure \ref{fig:domains_MREITpaper} for an exemplary setting of an electrode configuration and the domains. 
Later on, $\Ot$ will serve as contrast domain, where the unknown true conductivity is allowed to change from a constant background. $\Ott$ will be an intermediate domain between $\Ot$ and $\Omega$, which will be necessary for various theoretical arguments throughout this article.

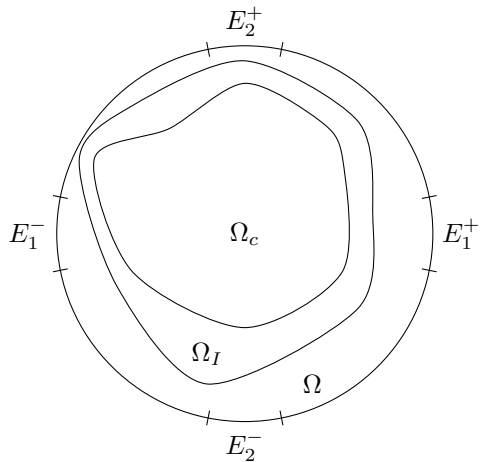
\begin{figure}[ht]
\label{fig:domains_MREITpaper}
\begin{center}
\begin{tikzpicture}
\foreach \ang in {7,9,15,17,23,25,31,1} {
  \draw (\ang * 180 / 16:2.4) -- (\ang * 180 / 16:2.6);
} 
\node [left] at (-2.5,0) {$E_1^-$};
\node [right] at (2.5,0) {$E_1^+$};
\node [above] at (0,2.5) {$E_2^+$};
\node [below] at (0,-2.5) {$E_2^-$};
\draw plot [smooth cycle] coordinates {(-1.5,-0.5) (-2,1) (-1,1.4) (0,2) (1,1.5) (1.3,1) (1.3,-0.5) (0,-1.25)} node at (0,0) {$\Ot$};
\draw plot [smooth cycle] coordinates {(-1.7,-0.7) (-2.2,1) (-1.5,1.7) (0,2.3) (1.5,1.5) (1.7,0.3) (1.5,-1) (-0.5,-2)} node at (-0.5,-1.6) {$\Ott$};
\draw (0,0) circle (2.5) node at (0.9,-2) {$\Omega$};
\end{tikzpicture}
\end{center}
\caption{Exemplary setting of the domains $\Ot\cc\Ott\cc\Omega\subset\R^2$ and attached electrodes $E_j^\pm$, $j=1,2$.}
\end{figure}

We consider the parameter space 
\begin{align*}
\mcP := \{ \sigma\in\Conea(\oO)\mid \sigma(x) > 0,\, x\in\oO\}
\end{align*}
and want to stress that this rather restrictive choice is made with sight on the convergence theory to be developed in sections \ref{subsec:IP_MREITpaper} \& \ref{subsec:conv_RBZ_MREITpaper}.

For the sake of completeness we include a proof of the fact, that a solution of the shunt model can be obtained as a scaled solution of a standard boundary value problem, see, e.g., \cite[lemma 2.1]{Jijun_Seo_Woo_MREITconv}.

\begin{lemma}
\label{lemma:fp_MREITpaper}
For $\sigma\in\mcP$, let $u_j^\sigma$ fulfill
\begin{subequations}
\label{eq:forward_problem_MREITpaper}
\begin{align}
\label{eq:forward_problem_PDE_MREITpaper}
\nabla\cdot(\sigma\nabla u_j^\sigma) = 0\quad\mbox{ in }\Omega \\
u_j^\sigma\vert_{E_j^+} = 1,\quad u_j^\sigma\vert_{E_j^-} = 0\\
\label{eq:forward_problem_NBnD_MREITpaper}
\sigma\nabla u_j^\sigma\cdot \mathbf{n} = 0\quad\mbox{ on }\partial\Omega\backslash \overline{E_j^+\cup E_j^-},
\end{align}
\end{subequations}
where $j=1,2$ specifies the active electrode pair $E_j^\pm$ and corresponding input current $I_j$. Then,
\begin{align}
\label{eq:scaled_Fsolve_MREITpaper}
\tilde{u}_j^\sigma = \frac{I_j}{\int_{E_j^+}\sigma\frac{\partial u_j^\sigma}{\partial\bn}\mathrm{d}s}u_j^\sigma
\end{align}
is a solution of the two-dimensional shunt model
\begin{subequations}
\label{eq:shunt_model_MREITpaper}
\begin{align}
\nabla\cdot(\sigma\nabla u) = 0\quad\mbox{in}\quad \Omega\\
I_j = \int_{E_j^+}\sigma\frac{\partial u}{\partial \bn} \mathrm{d}s = - \int_{E_j^-} \sigma\frac{\partial u}{\partial\bn} \mathrm{d}s,\quad 
\nabla u \times \bn = 0,\quad\mbox{on}\quad E_j^+\cup E_j^-,\\
\sigma\frac{\partial u}{\partial\bn} = 0\quad\mbox{on}\quad\partial\Omega\setminus\overline{\left(E_j^+\cup E_j^-\right)}.
\end{align}
\end{subequations}
\end{lemma}

\begin{proof}
$u_j^\sigma$ as a solution of \eqref{eq:forward_problem_MREITpaper} has Neumann boundary values $\sigma\frac{\partial u_j^\sigma}{\partial \bn}\in H^{-1/2}(\partial\Omega)$ such that
\begin{align*}
\int_\Omega \sigma \nabla u_j^\sigma \cdot \nabla w \mathrm{d}x = \int_{\partial \Omega} \sigma\frac{\partial u_j^\sigma}{\partial\bn} w \mathrm{d}s
\end{align*}
holds for all $w\in H^1(\Omega)$ and choosing $w \equiv 1$ as well as utilizing \eqref{eq:forward_problem_NBnD_MREITpaper} yields
\begin{align*}
\int_{E_j^+}\sigma\frac{\partial u_j^\sigma}{\partial\bn} \mathrm{d}s = -\int_{E_j^-}\sigma\frac{\partial u_j^\sigma}{\partial \bn}\mathrm{d}s.
\end{align*}
As a consequence, $\tilde{u}_j^\sigma = \frac{I_j}{\int_{E_j^+}\sigma\frac{\partial u_j^\sigma}{\partial\bn}\mathrm{d}s}u_j^\sigma$ fulfills the shunt model \eqref{eq:shunt_model_MREITpaper}.
\end{proof}

\begin{remark}
\label{rem:forward_problem_MREITpaper}
\begin{enumerate}[(i)]
\item Finding a solution $u_j^\sigma$ of \eqref{eq:forward_problem_MREITpaper} is equivalent to finding 
\begin{align*}
u\in H_{D_j}^1(\Omega):= \{u\in H^1(\Omega)\mid u\vert_{E_j^+} = 1,\, u\vert_{E_j^-} = 0\}
\end{align*}
 solving
\begin{subequations}
\label{eq:forward_problem_weakform_MREITpaper}
\begin{align}
b(u,v;\sigma) &= f(v),\quad\mbox{for all}\quad v\in H_0^1(\Omega),\\
b(u,v;\sigma) &:= \int_\Omega \sigma\nabla u \cdot \nabla v dx,\quad f(v) := 0.
\end{align}
\end{subequations}
Since $\sigma\in \mcP$, existence and uniqueness of a solution of \eqref{eq:forward_problem_weakform_MREITpaper} and therefore \eqref{eq:forward_problem_MREITpaper} follow via the Lax-Milgram theorem.
\item It is well known \cite{somersalo_cheney_isaacson_1992ExEindElectrodemodels} that the shunt model \eqref{eq:shunt_model_MREITpaper} omits a (up to an additive constant) unique solution. Therefore, the gradient of a solution of \eqref{eq:shunt_model_MREITpaper} is uniquely determined and equivalent to the gradient of \eqref{eq:scaled_Fsolve_MREITpaper}.
\item Whenever we refer to a solution of \eqref{eq:forward_problem_MREITpaper}, we refer to the scaled solution via \eqref{eq:scaled_Fsolve_MREITpaper} and will not write $\tilde{u}$ but $u$.
\end{enumerate}
\end{remark}
Before we formulate the inverse problem in the upcoming section, we gather various known regularity results and estimates for the solutions of mixed boundary value problems in the following lemma, see \cite[lemma 3.1]{Jijun_Seo_Woo_MREITerrest}. It is easy to see, that the general problem \eqref{eq:Lemma_proof_PDE_MREITpaper} in the upcoming lemma covers the forward problem \eqref{eq:forward_problem_MREITpaper}.
\begin{lemma}
\label{lemma:Lemma_proof_PDE_MREITpaper}
Denote by $\Gamma$ any relatively open $\Conea$-portion of $\partial\Omega$. For the boundary value problem
\begin{subequations}
\label{eq:Lemma_proof_PDE_MREITpaper}
\begin{align}
\nabla\cdot(\sigma\nabla u) &= \sigma g\quad\mbox{ in }\Omega \\
u\vert_\Gamma &= h \quad\mbox{ on }\Gamma\\
-\sigma\nabla u\cdot \mathbf{n} &= 0\quad\mbox{ on }\partial\Omega\backslash\overline{\Gamma},
\end{align}
\end{subequations}
with $\sigma\in \mcP$, and $g\in L^2(\Omega)$, $h\in H^{1/2}(\Gamma)$, it is $u\in H^2(\Ott)\cap H^1(\Omega)$.
\begin{enumerate}[(a)]
\item If $h=0$, the following estimates hold
\begin{align}
\label{eq:Lemma_proof_est3_MREITpaper}
\norm{u}_{H^2(\Ott)}&\leq C_1(\sigma)(\norm{u}_{L^2(\Omega)} + \norm{g}_{L^2(\Omega)}),\\
\label{eq:Lemma_proof_est4_MREITpaper}
\norm{u}_{H^1(\Omega)}&\leq C_2(\sigma) \norm{g}_{L^2(\Omega)}.
\end{align}
\item If $g\in C(\Omega)$, then $u\in C^{1,\alpha}(\Omega)$ with $\alpha\in (0,1)$ and 
\begin{align}
\label{eq:Lemma_proof_est5_MREITpaper}
\norm{\nabla u}_{C^{0,\alpha}(\Ot)}\leq C_3(\sigma)(\norm{u}_{C^{0,\alpha}(\Ott)} + \norm{g}_{C(\Ott)}).
\end{align}
\end{enumerate}
The functions $C_1(\sigma),\,C_2(\sigma),\,C_3(\sigma)$ are known bounded functions w.r.t. $\norm{\nabla\ln \sigma}_{C(\Omega)}$.
\end{lemma}
\begin{proof}
$u\in H^2(\Ott)\cap H^1(\Omega)$ and \eqref{eq:Lemma_proof_est3_MREITpaper} are direct consequences of \cite[thm. 8.8]{gilbarg_trudinger_PDEs}. 
\eqref{eq:Lemma_proof_est4_MREITpaper} can be obtained via the coercivity of the bilinearform and the continuity of the linearform of the variational problem associated with \eqref{eq:Lemma_proof_PDE_MREITpaper}. Finally, $u\in C^{1,\alpha}(\Omega)$ is obtained by \cite[cor. 8.36]{gilbarg_trudinger_PDEs} and the estimate \eqref{eq:Lemma_proof_est5_MREITpaper} is generated by \cite[thm. 8.32]{gilbarg_trudinger_PDEs} applied to $\Ot\subset\subset\Ott$.
\end{proof}
Do note, that all strong norms in this article , e.g. $\norm{\cdot}_{C(\Omega)}$ or $\norm{\cdot}_{\Ca(\Omega)}$, are always with respect to the closure of the specified domain.

\subsection{Inverse problem and properties}
\label{subsec:IP_MREITpaper}

For the remainder of this article, we assume that the unknown target conductivity $\sigs$ fulfills $\sigs\in \mcP$ with $\sigs\mid_{\oO\setminus\Ot} = \sigma_b$, where $\sigma_b >0$ is a known constant, and that the associated data sets $B_{z,\star}^1,\,B_{z,\star}^2$ are available and fulfill \eqref{eq:maxwell_relation_MREITpaper}, i.e.
\begin{align}
\label{eq:IP_data_MREITpaper}
\nabla^2 B_{z,\star}^j = \mu_0 \left(\frac{\partial \sigs}{\partial x} \frac{\partial u_j^\star}{\partial y} - \frac{\partial \sigs}{\partial y} \frac{\partial u_j^\star}{\partial x}\right),\quad j=1,2,
\end{align}
holds in a point-wise sense inside $\Omega$, where $u_j^\star$ denotes the solutions of \eqref{eq:forward_problem_MREITpaper} for $\sigma = \sigs$.
The \emph{inverse MREIT problem} then reads as follows:
\begin{align}
\label{eq:IP_MREITpaper}
\mbox{determine $\sigs$ from the knowledge of $B_{z,\star}^j$, $j = 1,2$.}
\end{align}
Motivated by \eqref{eq:maxwell_relation_MREITpaper}, we formulate the iteration sequence of the \emph{\Bz}  with initial guess $\sigma^0\in\mcP$.

\begin{procedure}[Iteration sequence]
\label{proc:BZ_MREITpaper}
\begin{enumerate}[1.]
\item Calculate the vector field 
\[
\mcV^{n+1}(\MR) := \begin{cases}
\frac{1}{\mu_0}\left[(\sigma^n(\MR)\A[\sigma^n](\MR))^{-1}\VEC{\nabla^2 B_{z,\star}^1(\MR)}{\nabla^2 B_{z,\star}^2(\MR)}\right],\quad \MR\in\Ott,\\
(0,0)^t,\quad \MR\in\Omega\setminus\Ott,
\end{cases}
\]
in $\Omega$, where $\A[\sigma^n](\MR) = \MATRIX{\frac{\partial u_1^n(\MR)}{\partial y}}{-\frac{\partial u_1^n(\MR)}{\partial x}}{\frac{\partial u_2^n(\MR)}{\partial y}}{-\frac{\partial u_2^n(\MR)}{\partial x}}$ and $u_j^n$ denotes the solution of the direct problem \eqref{eq:forward_problem_MREITpaper} for $\sigma = \sigma^n$.
\item Determine $\ln\sigma^{n+1}$ as the solution of 
\begin{align}
\label{eq:Iterationssequenz_MREITpaper}
\nabla^2 \ln\sigma^{n+1} = \nabla\cdot \mcV^{n+1}\quad\mbox{in }\Omega \qquad
\ln\sigma^{n+1} = \ln\sigs\quad\mbox{on }\partial\Omega.
\end{align}
\item Define the new iterate $\sigma^{n+1} := \exp(\ln\sigma^{n+1})>0$.
\end{enumerate}
\end{procedure}

\begin{remark}
\label{rem:Iteration_sequence_MREITpaper}
\begin{enumerate}[(i)]
\item We will often drop the dependency of $\mcV$ and $\A$ on $\MR\in\Omega$ and understand those quantities in a point-wise sense.
\item Procedure \ref{proc:BZ_MREITpaper} differs from previous formulations of the \Bz, see, e.g., \cite{HarmonicBz_mainpaper, HarmonicBz_2003b, Jijun_Seo_Woo_MREITconv, Jijun_Seo_Woo_MREITerrest}, by determining the new iterate as a solution of \eqref{eq:Iterationssequenz_MREITpaper}. We believe, that there exists a formulation equivalent to \eqref{eq:Iterationssequenz_MREITpaper} utilizing a suitable fundamental solution. Since this issue is not relevant for this work, we did not investigate it.
\item As mentioned in \cite[sec. 2.2]{Jijun_Seo_Woo_MREITerrest}, $\A[\sigma^n]$ is invertible in $\Ott$ (and procedure \ref{proc:BZ_MREITpaper} is well-defined) but does not need to be invertible up to the boundary. Furthermore, $\sigs\mid_{\Omega\setminus\Ott} = \sigma_b$ together with \eqref{eq:IP_data_MREITpaper} implies $\nabla^2 B_{z,\star}^1 = \nabla^2 B_{z,\star}^2 = 0$ in $\Omega\setminus\Ott$, such that it is reasonable to define $\mcV(\MR) = 0$ for $\MR\in\Omega\setminus\Ott$.
\item Regarding the uniqueness of the inverse problem \eqref{eq:IP_MREITpaper}, we refer to \cite[sec. 2.4.2]{seo_woo_2010_MREITcompilation}.
\item Throughout the article, the data $B_{z,\star}^1,\, B_{z,\star}^2$ is assumed to be known exactly. Denoising techniques have to be employed as soon as measurement noise is present in $B_{z,\star}^1,\, B_{z,\star}^2$ since the differentiation when obtaining $\nabla^2 B_{z,\star}^1,\, \nabla^2 B_{z,\star}^2$ will be sensitive to noise. We refer to \cite{HarmonicBz_mainpaper, HarmonicBz_2003b, seo_woo_2010_MREITcompilation, Sadleir_etal_MREIT_noise_estimation, JJliu2011noisy_MREIT_recon, seo_woo_2011_noniterativeBz} for various articles that examine the problem of noise in MREIT.
\item $\sigs$ is a fixed-point of procedure \ref{proc:BZ_MREITpaper} in the sense that if $\sigma^n = \sigs$, it is $\mcV^{n+1} = \nabla\ln\sigs$, $\ln\sigma^{n+1} = \ln\sigs$ and $\sigma^{n+1} = \sigs$. Nonetheless, we want to stress that aside from this correlation, $\mcV^{n+1} = 0$ in $\Omega\setminus\Ott$ does not imply that $\nabla\ln\sigma^{n+1} = 0$ in $\Omega\setminus\Ott$ as well.
\end{enumerate}
\end{remark}

It is the aim of this paper to extend the existing convergence theory \cite{Jijun_Seo_Woo_MREITconv, Jijun_Seo_Woo_MREITerrest}, such that an approximative solution of the direct problem \eqref{eq:forward_problem_MREITpaper}, e.g. a finite element approximation or an approximation of the type described in section \ref{sec:RBM_and_MREIT_MREITpaper}, can be used in procedure \ref{proc:BZ_MREITpaper} as well. To this end, we formulate the following iteration sequence of the \emph{approximative \Bz} with initial guess $\sigma^0\in\mcP$.

\begin{procedure}[Approximative iteration sequence]
\label{proc:approx_BZ_MREITpaper}
\begin{enumerate}[1.]
\item Calculate the approximative vector field 
\[
\mcV_N^{n+1}(\MR) := \begin{cases}
\frac{1}{\mu_0}\left[(\sigma^n(\MR)\A_N[\sigma^n](\MR))^{-1}\VEC{\nabla^2 B_{z,\star}^1(\MR)}{\nabla^2 B_{z,\star}^2(\MR)}\right],\quad \MR\in\Ott,\\
(0,0)^t,\quad \MR\in\Omega\setminus\Ott,
\end{cases}
\]
in $\Omega$, where $\A_N[\sigma^n](\MR) = \MATRIX{\frac{\partial u_{1,N}^n(\MR)}{\partial y}}{-\frac{\partial u_{1,N}^n(\MR)}{\partial x}}{\frac{\partial u_{2,N}^n(\MR)}{\partial y}}{-\frac{\partial u_{2,N}^n(\MR)}{\partial x}}$ and $u_{j,N}^n$ denotes the yet unspecified approximation to $u_j^n$, the exact solution of \eqref{eq:forward_problem_MREITpaper} for $\sigma = \sigma^n$.
\item Determine $\ln\sigma^{n+1}$ as the solution of 
\begin{align}
\label{eq:approx_Iterationssequenz_MREITpaper}
\nabla^2 \ln\sigma^{n+1} = \nabla\cdot \mcV_N^{n+1}\quad\mbox{in }\Omega \qquad
\ln\sigma^{n+1} = \ln\sigs\quad\mbox{on }\partial\Omega.
\end{align}
\item Define the new iterate $\sigma^{n+1} := \exp(\ln\sigma^{n+1})>0$.
\end{enumerate}
\end{procedure}

\begin{remark}
\label{rem:Approx_Itersequence_MREITpaper}
It is important to note that procedures \ref{proc:BZ_MREITpaper} \& \ref{proc:approx_BZ_MREITpaper} produce different sequences of iterates $\{\sigma^1,\,\sigma^2,\dots\}$. Whenever we refer to $u_j^n$, the solution of the direct problem \eqref{eq:forward_problem_MREITpaper} for $\sigma = \sigma^n$, it is meant with respect to the underlying procedure.
\end{remark}

Since $u_{j,N}^n$ is an approximative solution of \eqref{eq:forward_problem_MREITpaper}, the well-definedness of procedure \ref{proc:approx_BZ_MREITpaper} can not be obtained as in remark \ref{rem:Iteration_sequence_MREITpaper} and we make the following assumption.

\begin{assumption}
\label{ass:A_AN_invert_MREITpaper}
For the remainder of this article, we assume that procedure \ref{proc:approx_BZ_MREITpaper} is well-defined. Explicitly, we require the matrices $\A_N[\sigma^n]$ to be invertible inside of $\Ott$ for all $n=0,1,2,\dots$
\end{assumption}

This assumption is reasonable since $\A_N[\sigma^n] = \A[\sigma^n] + \left(\A_N[\sigma^n] - \A[\sigma^n]\right)$ and as long as the perturbation $\A_N[\sigma^n] - \A[\sigma^n]$ is small (e.g. when $u_{j,N}^n$ is a good approximation), the invertibility of $\A_N[\sigma^n]$ might hold through the respective property of $\A[\sigma^n]$.

With sight on the convergence theory in section \ref{subsec:conv_RBZ_MREITpaper}, we investigate the regularity of the iterates and the vector fields defined during procedures \ref{proc:BZ_MREITpaper} \& \ref{proc:approx_BZ_MREITpaper}.

\begin{theorem}
\label{thm:Iter_Reg_MREITpaper}
\begin{enumerate}[(a)]
\item The iterates $\sigma^{n+1}$ defined by procedure \ref{proc:BZ_MREITpaper} with initial guess $\sigma^0 \in\mcP$ fulfill $\sigma^{n+1}\in\mcP$ for $n=0,1,2,\dots$
\item As long as the approximations $u_{j,N}^{n}$ in procedure \ref{proc:approx_BZ_MREITpaper} fulfill $u_{j,N}^{n}\in \Conea(\Ott)$, the iterates $\sigma^{n+1}$ defined by procedure \ref{proc:approx_BZ_MREITpaper} with initial guess $\sigma^0 \in\mcP$ fulfill $\sigma^{n+1}\in\mcP$ for $n=0,1,2,\dots$
\end{enumerate}
\end{theorem}
\begin{proof}
\begin{enumerate}[(a)]
\item With $\sigs\in\mcP$ we can apply lemma \ref{lemma:Lemma_proof_PDE_MREITpaper} to derive $u_1^\star,\,u_2^\star\in\Conea(\Ott)$. Noticing that the product of h\"older-continuous functions is h\"older-contin\-uous, \eqref{eq:IP_data_MREITpaper} yields $\nabla^2 B_{z,\star}^1,\, \nabla^2 B_{z,\star}^2\in\Ca(\Ott)$. Repeating the same argument for $\sigma^0\in\mcP$ and associated forward solutions and combining it with the fact that for any $v\in \Ca(\Ott)$ with $v(\MR)\neq 0,\,\forall\MR\in\Ott$, it is $\frac{1}{v}\in\Ca(\Ott)$, we can deduce $\mcV^{1}\in\Ca(\Ott)$ (component-wise). 
Since $\sigs$ is already constant in $\oO\setminus\Ot$, it is $\nabla^2 B_{z,\star}^1 = \nabla^2 B_{z,\star}^1 = 0$ in $\oO\setminus\Ot$ and together with $\Ot\subset\subset \Ott\subset\subset\Omega$ it is actually $\mcV^{1}\in\Ca(\oO)$.
Since $\ln\sigma^{1}$ is defined as the solution of \eqref{eq:Iterationssequenz_MREITpaper}, its regularity is a consequence of the regularity of the right-hand-side and \cite[thm. 8.34]{gilbarg_trudinger_PDEs} yields the desired result. The remaining statement follows via induction.
\item The proof works analogously: the regularity assumption $u_{j,N}^{n}\in \Conea(\Ott)$ together with assumption \ref{ass:A_AN_invert_MREITpaper} yields the regularity of the vector field and the regularity of the iterates of procedure \ref{proc:approx_BZ_MREITpaper} follows from \cite[thm. 8.34]{gilbarg_trudinger_PDEs}
\end{enumerate}
\end{proof}

As a conclusion of this section, we gather properties of the iterates of procedure \ref{proc:approx_BZ_MREITpaper} which are interesting on their own and especially useful for the convergence proof in the upcoming section.

\begin{lemma}
\label{lemma:Lemma_iteration_properties_MREITpaper}
Let the approximations in procedure \ref{proc:approx_BZ_MREITpaper} fulfill $u_{j,N}^{n}\in \Conea(\Ott)$, such that theorem \ref{thm:Iter_Reg_MREITpaper} holds.
\begin{enumerate}[(a)]
\item There exists a constant $C^\dagger \geq 1$, that does not depend on $n$, such that
\begin{align}
\label{eq:proof_normjump_MREITpaper}
\norm{\ln\sigma^{n+1} - \ln\sigs}_{\Conea(\Ot)} \leq C^\dagger \norm{\mcV_N^{n+1} - \nabla\ln\sigs}_{\Ca(\Ot)},\, n = 0,1,2,\dots
\end{align}
\item There exists a constant $C^\ddagger \geq 1$, that does not depend on $n$, such that
\begin{align}
\label{eq:proof_domainjump_MREITpaper}
\norm{\ln\sigma^{n+1} - \ln\sigs}_{\Conea(\Omega)} \leq C^\ddagger \norm{\ln\sigma^{n+1} - \ln\sigs}_{\Conea(\Ot)},\, n = 0,1,2,\dots
\end{align}
\item It holds for $n = 0,1,2,\dots$: given an estimate $\norm{\ln\sigs - \ln\sigma^n}_{C^1(\Ot)} \leq K\epsilon^{n+1}$, for some $0<\epsilon <1$ and $K \geq 1$, there exists a constant $\tilde{K}\geq 1$, that does not depend on $n$, such that
\begin{align}
\label{eq:proof_somebound_MREITpaper}
\norm{\frac{\sigs - \sigma^n}{\sigma^n}}_{C^1(\Ot)} \leq \tilde{K}\epsilon^{n+1}.
\end{align}
\end{enumerate}
\end{lemma}

\begin{proof}
\begin{enumerate}[(a)]
\item It is obvious, that $\ln\sigma^{n+1} - \ln\sigs$ solves
\begin{align*}
\nabla^2 \ln\sigma^{n+1} - \nabla^2\ln\sigs &= \nabla\cdot \left(\mcV_N^{n+1} - \nabla\ln\sigs\right) \mbox{ in } \Omega,\\ 
\ln\sigma^{n+1} - \ln\sigs &= 0 \mbox{ on }\partial\Omega,
\end{align*}
and via \cite[(8.90)]{gilbarg_trudinger_PDEs} the estimate
\begin{align*}
\norm{\ln\sigma^{n+1} - \ln\sigs}_{\Conea(\Ot)} &\leq C \norm{\mcV_N^{n+1} - \nabla\ln\sigs}_{\Ca(\Omega)}
\end{align*}
holds, where $C$ is independent of $\mcV_N^{n+1}$ and thus $n$. Since $\sigs$ is constant in $\oO\setminus\Ot$, it is $\nabla^2 B_{z,\star}^1 = \nabla^2 B_{z,\star}^2 = 0$ and thus $\mcV_N^{n+1} = 0$ in $\oO\setminus\Ot$ as well.
Together with $C^\dagger := \max\{C,\, 1\} \geq 1$, it follows
\begin{align*}
\norm{\ln\sigma^{n+1} - \ln\sigs}_{\Conea(\Ot)} &\leq C^\dagger \norm{\mcV_N^{n+1} - \nabla\ln\sigs}_{\Ca(\Omega)}\\
&= C^\dagger \norm{\mcV_N^{n+1} - \nabla\ln\sigs}_{\Ca(\Ot)}.
\end{align*}
\item Since $\mcV_N^{n+1} = 0$ in $\Omega\setminus\Ot$, $e^{n+1} := \ln\sigma^{n+1} - \ln\sigs$ fulfills 
\begin{align*}
\nabla^2 e^{n+1} = 0\mbox{ in }\Omega\setminus\Ot, \qquad e^{n+1} = 0\mbox{ on }\partial\Omega, \qquad e^{n+1} = e^{n+1}\mbox{ on }\partial\Ot.
\end{align*}
Utilizing $\norm{e^{n+1}}_{\Conea(\Omega)} \leq \norm{e^{n+1}}_{\Conea(\Ot)} + \norm{e^{n+1}}_{\Conea(\Omega\setminus\Ot)}$ and applying \cite[(8.90)]{gilbarg_trudinger_PDEs} to $\norm{e^{n+1}}_{\Conea(\Omega\setminus\Ot)}$ yields
\begin{align*}
\norm{e^{n+1}}_{\Conea(\Omega)} \leq (1 + \tilde{C}) \norm{e^{n+1}}_{\Conea(\Ot)} =: C^\ddagger \norm{e^{n+1}}_{\Conea(\Ot)},
\end{align*}
where $\tilde{C}$ stems from \cite[(8.90)]{gilbarg_trudinger_PDEs} and is independent of $n$, such that $C^\ddagger \geq 1$ is also independent of $n$.
\item Regarding $\norm{\frac{\sigs - \sigma^n}{\sigma^n}}_{C(\Ot)}$, it is 
\[
\norm{\ln\sigs - \ln\sigma^n}_{C(\Ot)} \leq \norm{\ln\sigs - \ln\sigma^n}_{C^1(\Ot)} \leq K\epsilon^{n+1},
\]
such that
\begin{align}
\label{eq:lemma_sigmareg_2_MREITpaper}
e^{-K\epsilon^{n+1}}\leq\frac{\sigs}{\sigma^n}\leq e^{K\epsilon^{n+1}} \quad \Leftrightarrow \quad   e^{-K\epsilon^{n+1}}-1\leq\frac{\sigs}{\sigma^n}-1\leq e^{K\epsilon^{n+1}}-1
\end{align}
holds point-wise in $\overline{\Ot}$. Therefore, it is 
\begin{align*}
\norm{\frac{\sigs - \sigma^n}{\sigma^n}}_{C(\Ot)} \leq \max \{\vert e^{K\epsilon^{n+1}}-1\vert,\,\vert e^{-K\epsilon^{n+1}}-1\vert\}
\end{align*}
and applying the mean value theorem to $f_1(x) := e^{Kx}$ and $f_2(x) := e^{-Kx}$ yields the existence of $\xi_n^+,\,\xi_n^-\in(0,\epsilon^{n+1})\subset (0,\epsilon)$, such that
\begin{align*}
\vert e^{K\epsilon^{n+1}}-1\vert = \epsilon^{n+1} Ke^{K\xi_n^+}\quad\mbox{and}\quad \vert e^{-K\epsilon^{n+1}}-1\vert = \epsilon^{n+1} Ke^{-K\xi_n^-}.
\end{align*}
Since $Ke^{-K\xi_n^-} \leq Ke^{K\xi_n^+}\leq Ke^{K\epsilon}$ for $n = 0,1,2,\dots$, we define $\tilde{K} := Ke^{K\epsilon} \geq 1$ and conclude $\norm{\frac{\sigs - \sigma^n}{\sigma^n}}_{C(\Ot)} \leq \tilde{K}\epsilon^{n+1}$.\\
Regarding 
\begin{align*}
{\nabla\left(\frac{\sigs}{\sigma^n}-1\right)}_{C(\Ot)} = \norm{\nabla \frac{\sigs}{\sigma^n}}_{C(\Ot)},
\end{align*}
it is 
\begin{align*}
\norm{\nabla\left(\ln\left(\frac{\sigs}{\sigma^n}\right)\right)}_{C(\Ot)} \leq \norm{\ln\sigs - \ln\sigma^n}_{C^1(\Ot)} \leq K\epsilon^{n+1},
\end{align*}
such that
\begin{align*}
-K\epsilon^{n+1}\leq \frac{\sigma^n}{\sigs}\nabla\left(\frac{\sigs}{\sigma^n}\right)\leq K\epsilon^{n+1}
 \quad \Leftrightarrow \quad   
-\frac{\sigs}{\sigma^n}K\epsilon^{n+1}\leq\nabla\left(\frac{\sigs}{\sigma^n}\right)\leq \frac{\sigs}{\sigma^n}K\epsilon^{n+1}
\end{align*}
holds point-wise and component-wise. With \eqref{eq:lemma_sigmareg_2_MREITpaper}, it is $\norm{\frac{\sigs}{\sigma^n}}_{C(\Ot)}\leq e^{K\epsilon^{n+1}}$ and we conclude
\begin{align*}
\norm{\nabla\left(\frac{\sigs}{\sigma^n}\right)}_{C(\Ot)} \leq e^{K\epsilon^{n+1}}K\epsilon^{n+1} \leq \tilde{K} \epsilon^{n+1}
\end{align*}
and the statement follows.
\end{enumerate}
\end{proof}

Obviously, due to the first part of theorem \ref{thm:Iter_Reg_MREITpaper}, the results of lemma \ref{lemma:Lemma_iteration_properties_MREITpaper} can be obtained for the iterates of procedure \ref{proc:BZ_MREITpaper} as well.

\subsection{Convergence of the approximative Harmonic $B_z$ Algorithm\xspace}
\label{subsec:conv_RBZ_MREITpaper}

We gather supplementary results for the convergence theorem in the following lemma.

\begin{lemma}
\begin{enumerate}[(a)]
\item For a $\Conea$-domain $\Omega\subset\R^2$ and $a_1,\,a_2,\,a_3,\,a_4\in \Ca(\oO)$ 
\[
\norm{\MATRIX{a_1}{a_2}{a_3}{a_4}}_{\Ca(\Omega)} := 2\max_{i=1,2,3,4}\left\{\norm{a_i}_{\Ca(\Omega)}\right\}
\]
is a submultiplicative matrix norm that is consistent with the vector norm
\[
\norm{\VEC{a_1}{a_2}}_{\Ca(\Omega)} := \max\left\{\norm{a_1}_{\Ca(\Omega)},\norm{a_2}_{\Ca(\Omega)}\right\}.
\]
\item There exists a constant $C_\A >0$ such that $\norm{\A[\sigs]^{-1}}_{\Ca(\Ot)} \leq C_\A$.
\end{enumerate}
\end{lemma}
\begin{proof}
\begin{enumerate}[(a)]
\item The statement follows from standard arguments in matrix and vector norm theory such that we omit the proof.
\item As mentioned in the proof of theorem \ref{thm:Iter_Reg_MREITpaper} it is $u_1^\star,\,u_2^\star\in\Conea(\Omega)$ and with the notation $\MATRIX{a_1}{a_2}{a_3}{a_4} := \A[\sigs]$, it is
\begin{align*}
&\norm{\A[\sigs]^{-1}}_{\Ca(\Ot)}
\leq 2 \norm{\frac{1}{\det\A[\sigs]}}_{C(\Ot)}\max_{i=1,2,3,4}\left\{\norm{a_i}_{\Ca(\Ot)}\right\}\\
&= 2 \norm{\frac{1}{\det\A[\sigs]}}_{C(\Ot)}\max_{j=1,2}\norm{\nabla u_j^\star}_{\Ca(\Ot)},
\end{align*}
where $\max_{j=1,2}\norm{\nabla u_j^\star}_{\Ca(\Ot)}$ is finite. Furthermore, \cite[prop. 2.1]{Jijun_Seo_Woo_MREITerrest} which is based on \cite[prop. 2.10]{Alessandrini_Rosset_2004} yields the existence of a constant $\underline{\sigma}^\star>0$ such that $\norm{\frac{1}{\det\A[\sigs]}}_{C(\Ot)} \leq \frac{1}{\underline{\sigma}^\star}$. Therefore, we obtain the result
\begin{align}
\label{eq:bound_Asigs_inv_MREITpaper}
\norm{\A[\sigs]^{-1}}_{\Ca(\Ot)} \leq \frac{2}{\underline{\sigma}^\star}\max_{j=1,2}\norm{\nabla u_j^\star}_{\Ca(\Ot)} =: C_\A.
\end{align}
\end{enumerate}
\end{proof}

In the following, we present the main result of this article - the convergence result for procedures \ref{proc:BZ_MREITpaper} \& \ref{proc:approx_BZ_MREITpaper}. Do note, that this result is inspired by \cite[thm. 3.2]{Jijun_Seo_Woo_MREITerrest}.
\begin{theorem}
\label{thm:convergence_RBZ_MREITpaper}
Let $\sigs\in\mcP = \{ \sigma\in\Conea(\oO)\mid \sigma(x) > 0,\, x\in\oO\}$ with $\sigs\mid_{\oO\setminus\Ot} = \sigma_b$, $\sigma_b$ a known constant, and recall that $C^\dagger$ was introduced in lemma \ref{lemma:Lemma_iteration_properties_MREITpaper}.
\begin{enumerate}[(a)]
\item Considering procedure \ref{proc:BZ_MREITpaper}, we obtain the following convergence result. There exists an $\epsilon >0$, such that if $\norm{\nabla\ln\sigs}_{\Ca(\Omega)} < \epsilon$, the resulting sequence of iterates $\sigma^n,\, n=1,2,\dots$, with initial guess $\sigma^0 = \sigma_b$, satisfies
\begin{align*}
\norm{\ln\sigma^n - \ln\sigs}_{\Conea(\Ot)}\leq C^\dagger\left(\frac{1}{2}\right)^n\epsilon,\quad n=1,2,\dots
\end{align*}
\item Considering procedure \ref{proc:approx_BZ_MREITpaper}, we obtain the following convergence result. There exists an $\epsilon >0$, such that if $\norm{\nabla\ln\sigs}_{\Ca(\Omega)} < \epsilon$ and the approximations $u_{j,N}^n$ fulfill 
\begin{enumerate}[(i)]
\item $u_{j,N}^n\in \Conea(\Ott)$ \hfill (regularity condition)
\item $\norm{\nabla u_{j,N}^n - \nabla u_j^n}_{\Ca(\Ot)}\leq \frac{\epsilon^{n+1}}{2 C_{\A}}$ \hfill (quality condition)
\end{enumerate}
throughout procedure \ref{proc:approx_BZ_MREITpaper}, the resulting sequence of iterates $\sigma^n,\, n=1,2,\dots$, with initial guess $\sigma^0 = \sigma_b$, satisfies
\begin{align*}
\norm{\ln\sigma^n - \ln\sigs}_{\Conea(\Ot)}\leq C^\dagger\left(\frac{1}{2}\right)^n\epsilon,\quad n=1,2,\dots
\end{align*}
\end{enumerate}
\end{theorem}
\begin{proof}
\begin{enumerate}[(a)]
\item The exact forward solutions $u_j^n$ utilized in procedure \ref{proc:BZ_MREITpaper} fulfill both requirements of the second part of this theorem (the regularity stems from lemma \ref{lemma:Lemma_proof_PDE_MREITpaper} and the second requirement is trivial). Therefore, the first statement of this theorem is a consequence of the second statement.
\item We proof by induction: there exists an $\epsilon\in (0,1)$, such that if $\norm{\nabla\ln\sigs}_{\Ca(\Omega)} < \epsilon$ and both the regularity and the quality condition hold, there exists a $\theta < \frac{1}{2}$ depending on $\epsilon$, such that
\begin{align*}
\norm{\ln\sigma^n - \ln\sigs}_{\Conea(\Ot)}\leq C^\dagger \theta^n\epsilon\leq C^\dagger\left(\frac{1}{2}\right)^n\epsilon
\end{align*}
holds for $n=1,2,\dots,$ where $\epsilon$ and $\theta$ are fixed after the base case.\\
\emph{Base case ($n=0$):} 
Let $u_j^n$ and $u_j^\star$ denote the solutions of \eqref{eq:forward_problem_MREITpaper}
for $\sigma = \sigma^n$ and $\sigs$ respectively, where \eqref{eq:forward_problem_MREITpaper} was a special case of \eqref{eq:Lemma_proof_PDE_MREITpaper} with $g=0$, $\Gamma = E_j^+\cup E_j^-$ and adequately chosen $h$. Furthermore, we introduce the notation $e^0 := \ln\sigma^0 - \ln\sigs$, $w_j^0 := u_j^0 - u_j^\star$ and $w_{j,N}^0 := u_{j,N}^0 - u_j^\star$, where $w_j^0$ fulfills
\begin{subequations}
\label{eq:Convproof_PDE_MREITpaper}
\begin{align}
\nabla\cdot(\sigma^0 \nabla w_j^0) &= -\sigma^0\nabla e^0\cdot\nabla u_j^\star\quad\mbox{ in }\Omega \\
w_j^0\vert_{E_j^+} &= 0,\quad w_j^0\vert_{E_j^-} = 0\\
-\sigma^0 \nabla w_j^0\cdot \mathbf{n} = (\sigma^0-\sigs)\nabla u_j^\star\cdot\mathbf{n} &= 0 \quad\mbox{ on }\partial\Omega\backslash \overline{E_j^+\cup E_j^-}
\end{align}
\end{subequations}
since it is $\sigma^0 = \sigma_b$ and thus $\sigma^0 = \sigs$ on $\partial\Omega$.

To obtain the desired result, we want to utilize \eqref{eq:proof_normjump_MREITpaper} and derive a suitable estimate for $\norm{\mcV_N^1 - \nabla\ln\sigs}_{\Ca(\Ot)}$. 
According to procedure \ref{proc:approx_BZ_MREITpaper}, it is
\begin{align*}
\sigma^0\A_N[\sigma^0] \mcV_N^1 = \frac{1}{\mu_0}\VEC{\nabla^2 B_{z,\star}^1}{\nabla^2 B_{z,\star}^2}
\end{align*}
in $\Ot$. Introducing the notation $W_N^0 := \A_N[\sigma^0] - \A[\sigs]$ the above relation can be written as
\begin{align*}
(\sigma^0 I + \sigma^0\A[\sigs]^{-1} W_N^0)\mcV_N^1 = \frac{1}{\mu_0}\A[\sigs]^{-1}\VEC{\nabla^2 B_{z,\star}^1}{\nabla^2 B_{z,\star}^2} = \nabla\sigs,
\end{align*}
where $I\in\R^{2\times 2}$ is the identity matrix and the last equality was explained (for the logarithmic formulation of procedure \ref{proc:approx_BZ_MREITpaper}) in remark \ref{rem:Iteration_sequence_MREITpaper}.
Subtracting $(\sigma^0 I + \sigma^0\A[\sigs]^{-1} W_N^0)\nabla\ln\sigs$ on both sides and dividing by $\sigma^0$ yields
\begin{align}
\label{eq:Convproof_eq1_MREITpaper}
(I+\A[\sigs]^{-1}W_N^0)(\mcV_N^1 - \nabla\ln\sigs) = \left(\left(\frac{\sigs}{\sigma^0}-1\right)I - \A[\sigs]^{-1}W_N^0\right)\nabla\ln\sigs.
\end{align}
In order to derive the invertibility of $I+\A[\sigs]^{-1}W_N^0$ and also gain an upper bound on the matrix norm of its inverse via the Neumann series, we calculate
\begin{align*}
\norm{\A[\sigs]^{-1}W_N^0}_{\Ca(\Ot)} &\leq 2\norm{\A[\sigs]^{-1}}_{\Ca(\Ot)}\max_{j=1,2}\norm{\nabla w_{j,N}^0}_{\Ca(\Ot)}\\
&\leq  \underbrace{2\norm{\A[\sigs]^{-1}}_{\Ca(\Ot)}\max_{j=1,2}\norm{\nabla w_j^0}_{\Ca(\Ot)}}_{(*)}\\
&+ 
\underbrace{2\norm{\A[\sigs]^{-1}}_{\Ca(\Ot)}\max_{j=1,2}\norm{\nabla u_{j,N}^0 - \nabla u_j^0}_{\Ca(\Ot)}}_{(**)}.
\end{align*}
Defining $\hat{C}:= \max\{ C^\dagger,\, C^\ddagger \}\geq 1$, where $C^\dagger$ and $C^\ddagger$ have been introduced in lemma \ref{lemma:Lemma_iteration_properties_MREITpaper}, we make the initial choice of $\epsilon\in (0,\frac{1}{\hat{C}^2 +1})\cc (0,1)$.

Regarding $(**)$, we combine the quality condition $\norm{\nabla u_{j,N}^0 - \nabla  u_j^0}_{\Ca(\Ot)}\leq \frac{\epsilon}{2 C_{\A}}$ with \eqref{eq:bound_Asigs_inv_MREITpaper} to obtain
\begin{align}
\label{eq:Convproof_rb_err_MREITpaper}
2\norm{\A[\sigs]^{-1}}_{\Ca(\Ot)}\max_{j=1,2}\norm{\nabla u_{j,N}^0 - \nabla u_j^0}_{\Ca(\Ot)}
\leq 2 C_{\A} \frac{\epsilon}{2 C_{\A}} = \epsilon.
\end{align}
Regarding $(*)$, we want to derive an upper bound for $\norm{\nabla w_j^0}_{\Ca(\Ot)}$ containing $\epsilon$.
Since $w_j^0$ is a solution of \eqref{eq:Convproof_PDE_MREITpaper} which is covered by \eqref{eq:Lemma_proof_PDE_MREITpaper} with right hand side $g = \nabla e^0\cdot\nabla u_j^\star \in L^2(\Omega)$ (it is $u_j^\star\in H^1(\Omega)$ and $e^0\in \Conea(\oO)$), using \eqref{eq:Lemma_proof_est3_MREITpaper} and \eqref{eq:Lemma_proof_est4_MREITpaper} yields
\begin{align*}
\norm{w_j^0}_{H^2(\Ott)} &\leq C_1(\sigma^0)\left((\norm{\nabla e^0\cdot\nabla u_j^\star}_{L^2(\Omega)} + \norm{w_j^0}_{H^1(\Omega)}\right)\\ 
&\leq C_1(\sigma^0)(1+C_2(\sigma^0))\norm{\nabla e^0\cdot\nabla u_j^\star}_{L^2(\Omega)}\\
&\leq C_1(\sigma^0)(1+C_2(\sigma^0))\norm{\nabla e^0}_{C(\Omega)}\norm{ u_j^\star}_{H^1(\Omega)}.
\end{align*}
The Sobolev imbedding theorem \cite[thm. 4.12]{adams2003sobolev} yields the estimate $\norm{w_j^0}_{C^{0,\alpha}(\Ott)}\leq C_s\norm{w_j^0}_{H^2(\Ott)}$ with the imbedding constant $C_s$ and together with \eqref{eq:Lemma_proof_est5_MREITpaper} we obtain
\begin{align*}
&\norm{\nabla w_j^0}_{\Ca(\Ot)} \leq C_3(\sigma^0) [\norm{w_j^0}_{C^{0,\alpha}(\Ott)} + \norm{\nabla e^0\cdot \nabla u_j^\star}_{C(\Ott)}]\\
&\leq C_3(\sigma^0)[C_s \norm{ u_j^\star}_{H^1(\Omega)} C_1(\sigma^0)(1+C_2(\sigma^0)) + \norm{\nabla u_j^\star}_{C^{0,\alpha}(\Ott)}]\norm{\nabla e^0}_{C(\Omega)}.
\end{align*}
Note that \eqref{eq:Lemma_proof_est5_MREITpaper} is applicable here since $e^0\in \Conea(\oO)$ combined with $u_1^\star,\, u_2^\star\in C^{1,\alpha}(\Omega)$ it is $\nabla e^0\cdot \nabla u_j^\star\in C(\Omega)$.
At the same time $\norm{ u_j^\star}_{H^1(\Omega)}$ and $\norm{\nabla u_j^\star}_{C^{0,\alpha}(\Ott)}$ are finite and we denote by 
\begin{align*}
\tilde{G}(\sigma) := C_3(\sigma) [C_s \norm{ u_j^\star}_{H^1(\Omega)} C_1(\sigma)(1+C_2(\sigma))+\norm{\nabla u_j^\star}_{C^{0,\alpha}(\Ott)}]
\end{align*} 
a (due to lemma \ref{lemma:Lemma_proof_PDE_MREITpaper}) known function that only depends on $\norm{\nabla\ln\sigma}_{C(\Omega)}$. The expression
\begin{align*}
\sup_{\norm{\nabla\ln\sigma - \nabla\ln\sigs}_{C(\Omega)} \leq C^\dagger C^\ddagger \epsilon} \tilde{G}(\sigma)
\end{align*}
is well-defined since $\norm{\nabla\ln \sigma - \nabla\ln\sigs}_{C(\Omega)}\leq C^\dagger C^\ddagger\epsilon$ implies the boundedness of $\norm{\nabla\ln\sigma}_{C(\Omega)}$ via
\begin{align*}
\norm{\nabla\ln\sigma}_{C(\Omega)}\leq\norm{\nabla\ln\sigs}_{C(\Omega)} + \norm{\nabla\ln \sigma - \nabla\ln\sigs}_{C(\Omega)}\leq (1 + C^\dagger C^\ddagger)\epsilon \leq 1,
\end{align*}
where the last inequality stems from the initial choice of $\epsilon\in (0,\frac{1}{\hat{C}^2+1})$. Therefore, we define
\begin{align*}
\bar{G} := \sup_{\norm{\nabla\ln\sigma - \nabla\ln\sigs}_{C(\Omega)}\leq 1} \tilde{G}(\sigma),
\end{align*}
and remembering $\norm{\nabla\ln\sigs}_{\Ca(\Omega)} < \epsilon$ it is 
\begin{align*}
\norm{\nabla e^0}_{C(\Omega)} \leq \norm{\nabla \ln\sigs}_{\Ca(\Omega)} < \epsilon \leq  C^\dagger C^\ddagger \epsilon
\end{align*}
since $C^\dagger,\, C^\ddagger \geq 1$ and we assert $\tilde{G}(\sigma^0) \leq \bar{G}$. 
With the definition of $\hat{C}$, the above result yields the estimate
\begin{align}
\label{eq:Convproof_eq2_MREITpaper}
\norm{\nabla w_j^0}_{\Ca(\Ot)} \leq \tilde{G}(\sigma^0)\norm{\nabla e^0}_{C(\Omega)}
\leq \bar{G}\hat{C}^2\epsilon
\end{align}
and together with \eqref{eq:bound_Asigs_inv_MREITpaper}, we obtain
\begin{align}
\label{eq:Convproof_eq3_MREITpaper}
(*) = 2\norm{\A[\sigs]^{-1}}_{\Ca(\Ot)} \max_{j=1,2} \norm{\nabla w_j^0}_{\Ca(\Ot)} \leq 2 C_\A \bar{G}\hat{C}^2 \epsilon.
\end{align}
We refine the initial choice of $\epsilon\in(0,\frac{1}{\hat{C}^2 +1})$ and take $\epsilon$ small enough such that (with sight on \eqref{eq:Convproof_rb_err_MREITpaper}) $\max\{2C_\A \bar{G}\hat{C}^2\epsilon,\,\epsilon\} < \frac{1}{4}$. In total, it is
\begin{align*}
\norm{\A[\sigs]^{-1}W_N^0}_{\Ca(\Ot)} < \frac{1}{2}
\end{align*}
and the Neumann series for $I+\A[\sigs]^{-1}W_N^0$ is applicable. As a consequence, we obtain the following estimate which is based on \eqref{eq:Convproof_eq1_MREITpaper}, \eqref{eq:Convproof_rb_err_MREITpaper} as well as \eqref{eq:Convproof_eq3_MREITpaper}
\begin{align*}
\norm{\mcV_N^1 -\nabla\ln\sigs}&_{\Ca(\Ot)}\\
&\leq 2\left(2\norm{\frac{\sigs-\sigma^0}{\sigma^0}}_{\Ca(\Ot)}+(2 C_\A \bar{G}\hat{C}^2 +1)\epsilon\right)\norm{\nabla\ln\sigs}_{\Ca(\Ot)}.
\end{align*}
Using $\sigma^0 = \sigma_b$, we calculate
\begin{align*}
\norm{e^0}_{C(\Ot)} \leq diam(\Ot)\norm{\nabla e^0}_{C(\Ot)} = diam(\Ot)\norm{\nabla \ln\sigs}_{C(\Ot)} \leq diam(\Ot)\epsilon,
\end{align*}
such that $\norm{e^0}_{C^1(\Ot)}\leq (1+diam(\Ot))\epsilon$ and \eqref{eq:proof_somebound_MREITpaper} is applicable.
Together with \eqref{eq:proof_normjump_MREITpaper} and $\norm{\nabla\ln\sigs}_{\Ca(\Omega)} < \epsilon$, we obtain
\begin{align*}
\norm{\ln\sigma^1-\ln\sigs}_{\Conea(\Ot)}&\leq C^\dagger \norm{\mcV_N^1 - \nabla\ln\sigs}_{\Ca(\Ot)}\\
&\leq C^\dagger 2\left(2 \tilde{K} + 2 C_\A \bar{G}\hat{C}^2 +1\right)\epsilon^2,
\end{align*}
with $\tilde{K}\geq 1$ from \eqref{eq:proof_somebound_MREITpaper}.
So far, $\epsilon\in(0,\frac{1}{\hat{C}^2 +1})$ fulfills $\max\{2 C_\A \bar{G}\hat{C}^2 \epsilon,\,\epsilon\} < \frac{1}{4}$.
We finalize our choice of $\epsilon$ and chose it small enough, such that 
\begin{align*}
\theta := 2\left(2 \tilde{K} + 2 C_\A \bar{G}\hat{C}^2 +1\right)\epsilon < \frac{1}{2}
\end{align*}
and obtain
\begin{align*}
\norm{\ln\sigma^1 - \ln\sigs}_{\Conea(\Ot)}\leq C^\dagger \theta \epsilon < C^\dagger\left(\frac{1}{2}\right)\epsilon.
\end{align*}
\emph{Induction step $(n \rightarrow n+1)$:}
Let us assume that
\begin{align*}
\norm{\ln\sigma^k - \ln\sigs}_{\Conea(\Ot)} \leq C^\dagger \theta^k \epsilon < C^\dagger\left(\frac{1}{2}\right)^k \epsilon
\end{align*}
holds for $k =n$ and we want to verify the statement for $k = n+1$. We introduce the notation $e^n := \ln\sigma^n - \ln\sigs$ and the above induction hypothesis reads
\begin{align}
\label{eq:Convproof_eq4_MREITpaper}
\norm{e^n}_{\Conea(\Ot)}\leq C^\dagger\theta^n \epsilon = C^\dagger 2^n\left(2 \tilde{K} + 2 C_\A \bar{G}\hat{C}^2 +1\right)^n\epsilon^{n+1}.
\end{align}
Correspondingly, we introduce the notation $w_j^n = u_j^n - u_j^\star$ as well as $w_{j,N}^n = u_{j,N}^n - u_j^\star$ and since $\sigma^n = \sigs$ on $\partial\Omega$ due to \eqref{eq:approx_Iterationssequenz_MREITpaper}, $w_j^n$ meets
\begin{align*}
\nabla\cdot(\sigma^n \nabla w_j^n) = -\sigma^n\nabla e^n\cdot\nabla u_j^\star\quad\mbox{ in }\Omega \\
w_j^n\vert_{E_j^+} = 0,\quad w_j^n\vert_{E_j^-} = 0\\
-\sigma^n \nabla w_j^n\cdot \mathbf{n} = (\sigma^n - \sigs)\nabla u_j^\star \cdot\mathbf{n} = 0\quad\mbox{ on }\partial\Omega\backslash \overline{E_j^+\cup E_j^-}.
\end{align*}
Similar to \eqref{eq:Convproof_eq1_MREITpaper}, we obtain the equality
\begin{align*}
(I+\A[\sigs]^{-1}W_N^n)(\mcV_N^{n+1} - \nabla\ln\sigs) = \left(\left(\frac{\sigs}{\sigma^n}-1\right)I - \A[\sigs]^{-1}W_N^n\right)\nabla\ln\sigs,
\end{align*}
with $W_N^n := \A_N[\sigma^n] - \A[\sigs]$. 
As in the base case, we want to apply the Neumann series and therefore calculate
\begin{align*}
\norm{\A[\sigs]^{-1}W_N^n}&_{\Ca(\Ot)} \leq 2\norm{\A[\sigs]^{-1}}_{\Ca(\Ot)}\max_{j=1,2}\norm{\nabla w_{j,N}^n}_{\Ca(\Ot)}\\
&\leq \underbrace{2\norm{\A[\sigs]^{-1}}_{\Ca(\Ot)}\max_{j=1,2}\norm{\nabla w_j^n}_{\Ca(\Ot)}}_{(\diamond)}\\
&+ \underbrace{2\norm{\A[\sigs]^{-1}}_{\Ca(\Ot)}\max_{j=1,2}\norm{\nabla u_{j,N}^n - \nabla u_j^n}_{\Ca(\Ot)}}_{(\diamond\diamond)}.
\end{align*}
Regarding $(\diamond)$, due to the regularity condition $u_{j,N}^n\in \Conea(\Ott)$, theorem \ref{thm:Iter_Reg_MREITpaper} holds and with the regularity of $\sigma^n$ lemma \ref{lemma:Lemma_proof_PDE_MREITpaper} is applicable for $w_j^n$. Using the respective inequalities, similar to \eqref{eq:Convproof_eq2_MREITpaper}, we obtain 
\begin{align*}
\norm{\nabla w_j^n}_{\Ca(\Ot)} \leq \tilde{G}(\sigma^n)\norm{\nabla e^n}_{C(\Omega)}
\end{align*}
and \eqref{eq:proof_domainjump_MREITpaper} together with \eqref{eq:Convproof_eq4_MREITpaper} yields
\begin{align*}
\norm{\nabla e^n}_{C(\Omega)} \leq \norm{e^n}_{\Conea(\Omega)} \leq C^\ddagger \norm{e^n}_{\Conea(\Ot)} \leq C^\dagger C^\ddagger \epsilon < 1,
\end{align*}
which implies $\tilde{G}(\sigma^n)\leq\bar{G}$. In total, we obtain
\begin{align*}
\norm{\nabla w_j^n}_{C(\Ot)} 
\leq \bar{G}C^\ddagger C^\dagger \theta^n\epsilon.
\end{align*}
Using this, \eqref{eq:bound_Asigs_inv_MREITpaper} as well as $C^\dagger,\,C^\ddagger \leq \hat{C}$ and remembering the definition of $\theta$, we obtain
\begin{align*}
2\norm{\A[\sigs]^{-1}}_{\Ca(\Ot)}& \max_{j=1,2} \norm{\nabla w_j^n}_{\Ca(\Ot)}
\leq  2 C_\A \bar{G}\hat{C}^2 \theta^n\epsilon \leq \theta^{n+1} < \frac{1}{2^{n+1}} \leq \frac{1}{4}.
\end{align*}
Regarding $(\diamond\diamond)$, the quality condition $\norm{\nabla u_{j,N}^n - \nabla u_j^n}_{\Ca(\Ot)}\leq \frac{\epsilon^{n+1}}{2 C_{\A}}$ together with \eqref{eq:bound_Asigs_inv_MREITpaper} yields
\begin{align*}
2 \norm{\A[\sigs]^{-1}}_{\Ca(\Ot)}\max_{j=1,2}\norm{\nabla u_{j,N}^n - \nabla u_j^n}_{\Ca(\Ot)}
\leq 2 C_{\A} \frac{\epsilon^{n+1}}{2 C_{\A}} = \epsilon^{n+1}
\end{align*}
and we can assert ($\epsilon$ is at least smaller than $\frac{1}{4}$)
\begin{align*}
\norm{\A[\sigs]^{-1}W_N^n}_{\Ca(\Ot)} < \frac{1}{2}.
\end{align*}
Therefore, the Neumann series is once again applicable and together with the previous estimates as well as \eqref{eq:proof_somebound_MREITpaper} and the definition of $\theta$ we derive
\begin{align*}
&\norm{\mcV_N^{n+1}-\nabla\ln\sigs}_{\Ca(\Ot)}\\
&\leq 2\left(2 \tilde{K}\epsilon^{n+1} + 2\norm{\A[\sigs]^{-1}}_{\Ca(\Ot)} \max_{j=1,2}\norm{\nabla w_{j,N}^n}_{\Ca(\Ot)}\right)\norm{\nabla\ln\sigs}_{\Ca(\Ot)}\\
&\leq 2\epsilon^{n+1}\left(2 \tilde{K}+2 C_\A \bar{G}\hat{C}^2 2^n\left(2  \tilde{K} + 2 C_\A \bar{G}\hat{C}^2 +1\right)^n + 1\right)\epsilon\\
&\leq 2\epsilon^{n+1}2^n\left(2 \tilde{K} + 2 C_\A \bar{G}\hat{C}^2 +1\right)^{n+1}\epsilon = \theta^{n+1}\epsilon.
\end{align*}
Combining this with \eqref{eq:proof_normjump_MREITpaper} yields
\begin{align*}
\norm{\ln\sigma^{n+1}-\ln\sigs}_{\Conea(\Ot)}&\leq C^\dagger \norm{\mcV_N^{n+1} - \nabla\ln\sigs}_{\Ca(\Ot)}\\
&\leq C^\dagger \theta^{n+1}\epsilon < C^\dagger \left(\frac{1}{2}\right)^{n+1}\epsilon
\end{align*}
and the statement is correct for $k = n+1$.
\end{enumerate}
\end{proof}

\begin{remark}
\label{rem:conv_proof_MREITpaper}
Theorem \ref{thm:convergence_RBZ_MREITpaper} extends \cite[thm. 3.2]{Jijun_Seo_Woo_MREITerrest} in the following ways:
\begin{enumerate}[(i)]
\item The first part of theorem \ref{thm:convergence_RBZ_MREITpaper} replicates the statement of \cite[thm. 3.2]{Jijun_Seo_Woo_MREITerrest}. Do note, that as mentioned in remark \ref{rem:Iteration_sequence_MREITpaper} there is no direct correlation between either $\mcV^{n+1}$ or $\mcV_N^{n+1}$ and $\nabla\ln\sigma^{n+1}$ such that $\mcV^{n+1} = \mcV_N^{n+1} = 0$ in $\Omega\setminus\Ot$ can not be abused to generate a result on the iteration error in the background $\Omega\setminus\Ot$.
\item The second part of theorem \ref{thm:convergence_RBZ_MREITpaper} ensures actual numerical convergence of the \Bz. Remembering remark \ref{rem:Approx_Itersequence_MREITpaper}, $u_j^n$ in the quality condition is the solution of \eqref{eq:forward_problem_MREITpaper} for $\sigma = \sigma^n$ with $\sigma^n$ originating from procedure \ref{proc:approx_BZ_MREITpaper} such that the quality condition does relate to for instance a discretization error.
If the approximations $u_{j,N}^n$ are for example finite element approximations the regularity condition indicates what type of finite elements should be utilized to obtain a convergent numerical scheme and the quality condition specifies the required approximation quality, i.e. the fineness of the mesh.
\item The required bound on $\norm{\nabla\ln\sigs}_{\Ca(\Omega)}$ together with $\sigs\in\mcP$ translates to the contrast in $\sigs$ from the background $\sigma_b$ not being too large. In this sense the initial guess $\sigma^0 = \sigma_b$ can not be too far away from $\sigs$ and the convergence result acquired with theorem \ref{thm:convergence_RBZ_MREITpaper} has to be interpreted as a local convergence result.
\item Theorem \ref{thm:convergence_RBZ_MREITpaper} can be formulated for a broader class of initial guesses, possessing the properties necessary for the proof. Since the known background $\sigma_b$ is indeed the natural choice here (and fulfills the requirements), we did formulate the result with $\sigma^0 = \sigma_b$.
\end{enumerate}
\end{remark}

\section{Reduced basis methods for MREIT}
\label{sec:RBM_and_MREIT_MREITpaper}

Throughout section \ref{sec:MREIT_MREITpaper} an unspecified approximation to the forward solution $u_j^\sigma$ was used. In this section, we want to introduce a specific approximation via the reduced basis method, a model order reduction technique. Utilizing the reduced basis method we develop a novel algorithm, where it is the aim to speed-up the existing \Bz, and theorem \ref{thm:convergence_RBZ_MREITpaper} will guarantee the convergence of the new method, if the respective conditions are met. Finally, a numerical comparison of the novel algorithm and the existing one will be carried out.

\subsection{The reduced basis method}
\label{subsec:RBM_MREITpaper}

As introduced in remark \ref{rem:forward_problem_MREITpaper}, there are two different forward problems (one for each electrode configuration), to which we want to apply the reduced basis method. Since this introductory section wants to explain the basics of the reduced basis method, see, e.g., \cite{RB_master_paper, Ha14RBTut} for a detailed survey, we formulate it for only one of the two forward problems and choose w.l.o.g. $Y := H_{D_1}^1(\Omega)$ as the solution space in \eqref{eq:forward_problem_weakform_MREITpaper} and write $u$ or $u^\sigma$ instead of $u_1^\sigma$ whenever we refer to  \eqref{eq:forward_problem_weakform_MREITpaper} in this section. Of course, the method can be analogously formulated for the second solution space $H_{D_2}^1(\Omega)$.

The aim of the reduced basis method is the construction of an accurate reduced basis approximation $u_{N}^\sigma$ of $u^\sigma$, the solution of \eqref{eq:forward_problem_weakform_MREITpaper}, where $u_{N}^\sigma\in Y_N\subset Y$, the \emph{reduced basis space} with $\dim Y_N = N\in\N$ and $N\ll\infty$.
Typically $Y_N$ will consist of so called \emph{snapshots} that are solutions of \eqref{eq:forward_problem_weakform_MREITpaper} for \emph{meaningful} parameters.
We will discuss our method of constructing $Y_N$ in the upcoming section and assume its existence in this section.

\begin{definition}
\label{def:reducedprob_MREITpaper}
Given the forward problem \eqref{eq:forward_problem_weakform_MREITpaper} and a reduced basis space $Y_N \subset Y$, with $\dim Y_N = N$ and basis $\Psi_N := \{\psi_1, \dots , \psi_N\}$, we define the \emph{reduced forward problem}: for $\sigma\in \mcP$ find $u_N^\sigma\in Y_N$ the solution of
\begin{align}
\label{eq:reducedfp_MREITpaper}
b(u_N,v;\sigma) = f(v),\quad\mbox{for all}\quad v\in Y_N.
\end{align}
We call $u_N^\sigma$ the \emph{reduced basis approximation} and will often write $u_N$ instead.
\end{definition}

Since \eqref{eq:reducedfp_MREITpaper} is simply the Galerkin-projection of \eqref{eq:forward_problem_weakform_MREITpaper} onto $Y_N$, a closed subspace of $Y$, existence and uniqueness of a solution of \eqref{eq:reducedfp_MREITpaper} follow from the properties of \eqref{eq:forward_problem_weakform_MREITpaper}.

To give a better impression of \eqref{eq:reducedfp_MREITpaper} and insight on its numerical implementation, we define the \emph{discrete reduced forward problem}.

\begin{proposition}
\label{prop:discrete_RB_MREITpaper}
For a given reduced forward problem \eqref{eq:reducedfp_MREITpaper} and $\sigma\in\mcP$, we define
\begin{align*}
\mathbf{B}_N(\sigma) := \left(b(\psi_j,\psi_i;\sigma)\right)_{i,j=1}^N\in\R^{N\times N},\quad \mathbf{f}_N := \left(f(\psi_i)\right)_{i=1}^N\in\R^N.
\end{align*}
Solving the linear system 
\begin{align}
\label{eq:reduceddiscfp_MREITpaper}
\mathbf{B}_N(\sigma) \mathbf{u}_N^\sigma = \mathbf{f}_N
\end{align} 
for $\mathbf{u}_N^\sigma = \left( u_{N,i}\right)_{i=1}^N\in\R^N$, we can obtain the solution of \eqref{eq:reducedfp_MREITpaper} via $u_N^\sigma = \sum_{i=1}^N u_{N,i}\psi_i$.
\end{proposition}

Regarding the stability of \eqref{eq:reduceddiscfp_MREITpaper}, if the reduced basis $\Psi_N$ is orthonormal, it holds $\mathrm{cond}(\mathbf{B}_N(\sigma))\leq \frac{\gamma(\sigma)}{\alpha(\sigma)}$ independent of $N$, with $\alpha(\sigma)$ and $\gamma(\sigma)$ being the coercivity and continuity constants of the bilinear form $b$.

From a numerical viewpoint $\mathbf{B}_N(\sigma)$ will not be sparse, but since $N$ is usually very small, the solution of \eqref{eq:reduceddiscfp_MREITpaper} is still very cheap compared to e.g. the computation of a finite element approximation of the full forward problem \eqref{eq:forward_problem_weakform_MREITpaper}.

We formulate two simple, but important properties of the reduced basis method: the well-known rigorous error estimator for the reduced basis error, here measured in the $H^1$-norm $\norm{u - u_N}_{H^1(\Omega)}$, and the \emph{reproduction of solutions}.
\begin{lemma}
\label{lemma:RBMproperties_MREITpaper}
\begin{enumerate}[(a)]
\item For $\sigma\in \mcP$ we define the residual $r(\cdot ; \sigma)\in Y'$ via $r(v;\sigma) := f(v) - b(u_N,v;\sigma),\, v\in Y$ and let $v_r\in Y$ denote the Riesz-representative of $r(\cdot ;\sigma)$, i.e.,
\begin{align*}
\sp{v_r}{v}_{H^1(\Omega)} = r(v;\sigma),\, v\in Y,\quad \norm{v_r}_{H^1(\Omega)} = \norm{r(\cdot ; \sigma)}_{Y'}.
\end{align*}
Then, the error $u - u_N\in Y$ is bounded for all $\sigma\in \mcP$ by
\begin{align*}
\norm{u - u_N}_{H^1(\Omega)} \leq \Delta_N(\sigma) := \frac{\norm{v_r}_{H^1(\Omega)}}{\alpha(\sigma)}.
\end{align*}
\item For $\sigma\in \mcP$, let $u$, $u_N$ be solutions of \eqref{eq:forward_problem_weakform_MREITpaper} and \eqref{eq:reducedfp_MREITpaper} and $\mathbf{e}_i\in\R^N$ the $i$-th unit vector. Then the following holds
\begin{enumerate}[(i)]
\item if $u\in Y_N$ then $u_N = u$.
\item if $u = \psi_i\in \Psi_N$ then $\mathbf{u}_N = \mathbf{e}_i\in\R^N$ in proposition \ref{prop:discrete_RB_MREITpaper}.
\end{enumerate}
\end{enumerate}
\end{lemma}

\begin{proof}
\begin{enumerate}[(a)]
\item See, e.g., \cite{RB_master_paper} or \cite[prop. $2.15$ \& $2.19$]{Ha14RBTut}.
\item Immediately follows from \eqref{eq:forward_problem_weakform_MREITpaper} and \eqref{eq:reducedfp_MREITpaper}, see, e.g., \cite[prop. $2.16$]{Ha14RBTut}.
\end{enumerate}
\end{proof}

We want to close this introductory section by commenting on the reduced basis method in a  numerical setting where a fully discretized forward problem (including a finite dimensional parameter space) is given instead of \eqref{eq:forward_problem_weakform_MREITpaper}.
\begin{remark}
\label{rem:RBM_MREITpaper} 
\begin{enumerate}[(i)] 
\item The discrete forward solution takes over the role of $u$ in this setting and the reduced basis solutions approximate this discrete forward solution. As a consequence, the reduced basis error does not incorporate the approximation error that is made by the discrete forward problem and the error estimator does not include this error. This is a usual occurrence in reduced basis methods and it is assumed that the discrete forward problem is chosen well enough such that its approximation error is negligible.
\item Given a discretized forward problem, the method (including the error estimator) can efficiently be implemented utilizing an offline/online decomposition, see, e.g., \cite[sec. $7.1.3$]{RB_master_paper}, \cite[sec. 2.5]{Ha14RBTut} or \cite[sec. 3.2]{G_Harrach_Haasdonk}, such that the reduced basis approximation and the error estimator can be rapidly computed and a considerable speed-up is achieved. In order to keep the length of this manuscript healthy, we chose not to explain this in detail.
\end{enumerate}
\end{remark}

\subsection{The \RBz (RBZ-Algorithm)}
\label{subsec:RBZ_MREITpaper}

Let us combine procedure \ref{proc:BZ_MREITpaper} with a suitable termination criterion as a starting point for the development of our new method. Motivated by the fixed-point discussion in remark \ref{rem:Iteration_sequence_MREITpaper} and the convergence result in theorem \ref{thm:convergence_RBZ_MREITpaper}, we choose the logarithmic iteration error as termination criterion and formulate the following \emph{\Bz}.

\begin{algorithm}
\caption{\Bz($\sigma^0 = \sigma_b,\mu_0,\varepsilon,\nabla^2 B_{z,\star}^1,\nabla^2 B_{z,\star}^2$)}
\begin{algorithmic}[1]
\label{algo:BZ_MREITpaper}
\STATE{$n = 0$}
\REPEAT
\STATE{For all $\MR\in\Omega$, calculate the vector field
\[
\mcV^{n+1}(\MR) := \begin{cases}
\frac{1}{\mu_0}\left[(\sigma^n(\MR)\A[\sigma^n](\MR))^{-1}\VEC{\nabla^2 B_{z,\star}^1(\MR)}{\nabla^2 B_{z,\star}^2(\MR)}\right],\quad \MR\in\Ott,\\
(0,0)^t,\quad \MR\in\Omega\setminus\Ott.
\end{cases}
\]
}
\STATE{Calculate $\ln\sigma^{n+1}$ as the solution of \eqref{eq:Iterationssequenz_MREITpaper}.}
\STATE{$\sigma^{n+1} := \exp(\ln\sigma^{n+1})$}
\STATE{$n=n+1$}
\UNTIL{$\norm{\ln\sigma^n - \ln\sigma^{n-1}}_{C(\Omega)} < \varepsilon$}
\RETURN{$\sigma_{BZ} = \sigma^n$}
\end{algorithmic}
\end{algorithm}

\begin{remark}
\label{rem:BZ_algo_MREITpaper}
\begin{enumerate}[(i)]
\item If $\sigs$ fulfills the requirements of theorem \ref{thm:convergence_RBZ_MREITpaper}, algorithm \ref{algo:BZ_MREITpaper} terminates. 
As mentioned in the end of section \ref{subsec:IP_MREITpaper}, lemma \ref{lemma:Lemma_iteration_properties_MREITpaper} holds for procedure \ref{proc:BZ_MREITpaper} as well.
Applying the triangle inequality and \eqref{eq:proof_domainjump_MREITpaper} yields
\begin{align*}
\norm{\ln\sigma^n - \ln\sigma^{n-1}}&_{C(\Omega)}  \leq \norm{\ln\sigma^n - \ln\sigma^{n-1}}_{\Conea(\Omega)} \\
&\leq \norm{\ln\sigma^n - \ln\sigs}_{\Conea(\Omega)} + \norm{\ln\sigma^{n-1} - \ln\sigs}_{\Conea(\Omega)}\\
&\leq C^\ddagger \left(\norm{\ln\sigma^n - \ln\sigs}_{\Conea(\Ot)} + \norm{\ln\sigma^{n-1} - \ln\sigs}_{\Conea(\Ot)}\right)
\end{align*}
and it follows from theorem \ref{thm:convergence_RBZ_MREITpaper} that the last expression goes to zero as $n$ goes to infinity.
Therefore, the chosen termination criterion is reasonable although, keeping the convergence result of theorem \ref{thm:convergence_RBZ_MREITpaper} in mind, simply running a fixed amount of repeat-loop iterations would also yield decent results.
\item Alternatively, an efficiently computable error estimator for $\norm{\ln\sigma^n - \ln\sigs}_{\Conea(\Omega)}$ could be used as termination criterion, see \cite[thm 2.1]{Jijun_Seo_Woo_MREITerrest} for a first result in that direction.
\end{enumerate}
\end{remark}

It is our intention to develop a faster version of algorithm \ref{algo:BZ_MREITpaper} involving the reduced basis method presented in section \ref{subsec:RBM_MREITpaper}, where the reconstruction with the new algorithm should retain its quality compared to algorithm \ref{algo:BZ_MREITpaper}. Using a qualitative and cheap approximative forward solution in order to speed-up the whole algorithm is intuitive, since the computationally expensive part of each iteration of algorithm \ref{algo:BZ_MREITpaper} is the computation of the two solutions of \eqref{eq:forward_problem_weakform_MREITpaper} (one per electrode configuration) involved in the matrix $\A[\sigma^n]$. One way to apply the reduced basis method would be, what we call the \emph{direct approach}:
\begin{enumerate}
\item For each forward problem \eqref{eq:forward_problem_weakform_MREITpaper} construct a \emph{global reduced basis space}, e.g. via the classical greedy-algorithm, see, e.g., \cite{veroy2003posteriori, RB_master_paper, Ha14RBTut}, where it is the aim of a global reduced basis space to provide accurate approximations for every parameter in the parameter domain (the desired accuracy is given by the quality condition in theorem \ref{thm:convergence_RBZ_MREITpaper}).
\item In each step of algorithm \ref{algo:BZ_MREITpaper}, replace the forward solutions of \eqref{eq:forward_problem_weakform_MREITpaper} by the corresponding reduced basis approximations.
\end{enumerate}
The offline/online decomposition mentioned in remark \ref{rem:RBM_MREITpaper} would guarantee the desired speed-up and as long as the quality condition in theorem \ref{thm:convergence_RBZ_MREITpaper} is fulfilled the convergence would also be guaranteed since snapshot-based reduced basis spaces inherit the regularity from the snapshots. 

This direct approach has successfully been applied to inverse problems with a low-dimensional parameter space, see, e.g., \cite{nguyenrozza2009reduced, nguyenPHDThesis, HoangDentalTissuereduced}.
In the imaging context of this article however, we want to recover high-resolution images of the unknown conductivity, such that a potential discrete parameter domain in a numerical setting would be very high-dimensional. 
This high-dimensionality limits the applicability of this direct approach since in general it is impossible to construct a well-approximating global reduced basis space for a complex high-dimensional parameter domain. We refer to \cite[rem. 2.6]{Ha14RBTut} for more details on this discussion.

To overcome this issue of dimensionality and to be able to tackle parameter spaces of arbitrary dimension, we propose an approach which aims at constructing a \emph{locally approximating} reduced basis space. We outlined this local approach for the nonlinear Landweber method in \cite{G_Harrach_Haasdonk} and want to mention that it is based on ideas developed in \cite{Druskin_Zaslavski, cuiWilcox2014datadriveninversion, lass2014PHDThesis, zahr2014progressive}.
The key idea is to simultaneously solve the inverse problem as well as adaptively enrich and therefore construct the reduced basis space:
\begin{enumerate}
\item Start with two given reduced basis spaces $Y_{N,1} = \mathrm{span}\{\Psi_{N,1}\}$ and $Y_{N,2} = \mathrm{span}\{\Psi_{N,2}\}$ (one per forward problem).
\item Project the reconstruction algorithm on this set of reduced basis spaces. In our case this leads to procedure \ref{proc:approx_BZ_MREITpaper} where $u_{1,N}^n$ and $u_{2,N}^n$ are the reduced basis approximations introduced in definition \ref{def:reducedprob_MREITpaper} for $Y_{N,1}$ and $Y_{N,2}$ respectively.
\item Run this projected algorithm until either the current iterate is accepted as a solution to the inverse problem ($\leadsto$ termination) or the approximation quality of the reduced spaces is not trusted anymore ($\leadsto$ step 4).
\item If the approximation quality of the reduced spaces was insufficient, the current iterate is utilized to generate snapshots for the enrichment of $Y_{N,1}$ and $Y_{N,2}$ ($\leadsto$ step $2$).
\end{enumerate} 
Basically, we abuse the ability of our inversion algorithm to find parameter values that approach the exact solution in order to determine \emph{relevant} parameters for which we can include the snapshots into our reduced basis spaces. Based on these ideas, we formulate the following \emph{\RBz} (RBZ).

\begin{algorithm}
\caption{RBZ($\sigma^0 = \sigma_b,\mu_0,\varepsilon_1,\varepsilon_2,\Psi_{N,1},\Psi_{N,2},\nabla^2 B_{z,\star}^1,\nabla^2 B_{z,\star}^2$)}
\begin{algorithmic}[1]
\label{algo:RBZ_MREITpaper}
\STATE{$n = 0$, $Y_{N,1} = \mathrm{span}\{\Psi_{N,1}\}$, $Y_{N,2} = \mathrm{span}\{\Psi_{N,2}\}$}
\REPEAT
\STATE{$\Psi_{N,1} = \Psi_{N,1}\cup\{u_1^n\}$, $\Psi_{N,2} = \Psi_{N,2}\cup\{u_2^n\}$}
\STATE{$Y_{N,1} = \mathrm{span}\{\Psi_{N,1}\}$, $Y_{N,2} = \mathrm{span}\{\Psi_{N,2}\}$}
\REPEAT
\STATE{For all $\MR\in\Omega$, calculate the vector field
\[
\mcV_N^{n+1}(\MR) := \begin{cases}
\frac{1}{\mu_0}\left[(\sigma^n(\MR)\A_N[\sigma^n](\MR))^{-1}\VEC{\nabla^2 B_{z,\star}^1(\MR)}{\nabla^2 B_{z,\star}^2(\MR)}\right],\quad \MR\in\Ott,\\
(0,0)^t,\quad \MR\in\Omega\setminus\Ott.
\end{cases}
\]
}
\STATE{Calculate $\ln\sigma^{n+1}$ as the solution of \eqref{eq:approx_Iterationssequenz_MREITpaper}.}
\STATE{$\sigma^{n+1} := \exp(\ln\sigma^{n+1})$}
\STATE{$n = n+1$}
\UNTIL{$\norm{\ln\sigma^n - \ln\sigma^{n-1}}_{C(\Omega)} < \varepsilon_1$} \OR {$\min_{j=1,2}\{\norm{\nabla u_{j,N}^n - \nabla u_j^n}_{\Ca(\Omega)}\} > \varepsilon_2^{n+1}$}
\UNTIL{$\norm{\ln\sigma^n - \ln\sigma^{n-1}}_{C(\Omega)} < \varepsilon_1$}
\RETURN{$\sigma_{RBZ} = \sigma^n$}
\end{algorithmic}
\end{algorithm}

\begin{remark}
\label{rem:RBZ_MREITpaper}
\begin{enumerate}[(i)]
\item The initial reduced bases $\Psi_{N,1}$ and $\Psi_{N,2}$ in algorithm \ref{algo:RBZ_MREITpaper} can be empty since they are directly enriched with the snapshots for the initial guess $\sigma^0$. Furthermore, $\Psi_{N,1},\,\Psi_{N,2}$ are always orthonormalized to ensure numerical stability.
\item If $\sigs$ fulfills the requirements of theorem \ref{thm:convergence_RBZ_MREITpaper}, algorithm \ref{algo:RBZ_MREITpaper} terminates by an argument similar to remark \ref{rem:BZ_algo_MREITpaper}, where $\varepsilon_1$ and $\varepsilon_2$ have to be chosen accordingly.
\item The formulation in algorithm \ref{algo:RBZ_MREITpaper} is tailored around the convergence result of theorem \ref{thm:convergence_RBZ_MREITpaper} and is not suitable for a numerical implementation: on the one hand it is difficult to handle the h\"older-norms numerically and on the other hand the criterion $\min_{j=1,2}\{\norm{\nabla u_{j,N}^n - \nabla u_j^n}_{\Ca(\Ot)}\} > \varepsilon_2^{n+1}$, although ensuring the quality condition of theorem \ref{thm:convergence_RBZ_MREITpaper} if $\varepsilon_2$ is chosen appropriately, is not a reasonable criterion from a reduced basis point of view. In order to check the criterion, one would have to compute the (computationally expensive) exact forward solutions $u_j^n$, which would defeat the purpose of a model order reduction approach.
\item Therefore, we will explain in the upcoming section the simplified version of the algorithm (utilizing the error estimator introduced in lemma \ref{lemma:RBMproperties_MREITpaper}) that was used for the numerical experiments.
\end{enumerate}
\end{remark}

\subsection{Numerical Experiments}
\label{subsec:Numerics_MREITpaper}

It is our intention to perform a short numerical comparison of algorithms \ref{algo:BZ_MREITpaper} \& \ref{algo:RBZ_MREITpaper}, including one reconstruction from exact data and one reconstruction from noisy data, where the reconstruction quality as well as the computational time will be compared. It is not our intention to verify the theoretical results from theorem \ref{thm:convergence_RBZ_MREITpaper}, such that the numerical setting below does not claim to be consistent with the requirements of the theory developed in section \ref{sec:MREIT_MREITpaper}.

The to be reconstructed conductivity is a piecewise linear approximation of the shepp-logan-phantom with $260\times260$ pixels and $1$ added to the grayscale values to ensure coercivity. It is visualized in the top right of figure \ref{fig:compare_reconstructions_MREITpaper}, where in the top left the initial value $\sigma^0 = \sigma_b = 1$ can be seen.
To match the phantom, we choose $\Omega := [-1,1]^2$ and $\Ott := \{(x,y)\mid \sqrt{x^2+y^2} < 0.95 \}\cc\Omega$ with electrode pairs $E_1^\pm := \{(\pm 1,y)\mid \vert y \vert < 0.1\}$, $E_2^\pm := \{(x,\pm 1)\mid \vert x \vert < 0.1\}$.
For simplicity, it is $\mu_0 = 1$ and the scaling introduced in lemma \ref{lemma:fp_MREITpaper} is not performed.
The PDEs \eqref{eq:forward_problem_weakform_MREITpaper}, \eqref{eq:Iterationssequenz_MREITpaper} and \eqref{eq:approx_Iterationssequenz_MREITpaper} are discretized on a triangular mesh with $135200$ elements and piecewise linear finite elements are utilized for all three PDEs. As a result, the data $\nabla^2 B_{z,\star}^1,\,\nabla^2 B_{z,\star}^2$, which is generated synthetically via \eqref{eq:IP_data_MREITpaper} where Comsol\textsuperscript{\textregistered} is used for the involved PDE solutions to prevent inverse crime, is piecewise constant on the grid. 
The noisy data set in this comparison is generated by (triangle-wise) adding $10\%$ relative gaussian noise to $\nabla^2 B_{z,\star}^1,\,\nabla^2 B_{z,\star}^2$ (wherever $\nabla^2 B_{z,\star}^1,\,\nabla^2 B_{z,\star}^2$ is equal to zero, the average absolute value of $\nabla^2 B_{z,\star}^1$ or $\nabla^2 B_{z,\star}^2$ is taken as the reference value for the gaussian relative noise). We want to emphasize that although testing the reconstruction algorithms for robustness this circumvents the problem described in remark \ref{rem:Iteration_sequence_MREITpaper} that occurs when differentiating actual noisy $B_z$ data (e.g. real-world measurements).
In order to ensure the approximation quality of the reduced basis spaces in algorithm \ref{algo:RBZ_MREITpaper}, $\min_{j=1,2}\{\norm{\nabla u_{j,N}^n - \nabla u_j^n}_{\Ca(\Omega)}\} > \varepsilon_2^{n+1}$ is replaced by $\min_{j=1,2}\left\{\Delta_{1,N}(\sigma^n),\, \Delta_{2,N}(\sigma^n)\right\} > \varepsilon_2$, where $\Delta_{j,N}(\sigma^n)$ was the rigorous reduced basis error estimator introduced in lemma \ref{lemma:RBMproperties_MREITpaper} with $j$ indicating the underlying reduced basis space $Y_{N,1}$ or $Y_{N,2}$.
This termination criterion is computationally cheap to evaluate and the error estimator in the $H^1(\Omega)$ norm should contain some derivative information.
Finally, it is $\varepsilon = \varepsilon_1 = 10^{-6}$ as the acceptance tolerance in algorithms \ref{algo:BZ_MREITpaper} \& \ref{algo:RBZ_MREITpaper} and $\varepsilon_2 = 10^{-3}$ for the new termination criterion.
The numerical experiment is performed using  Matlab\textsuperscript{\textregistered} in conjunction with the libraries RBmatlab and KerMor which both can be found online\footnote{http://www.ians.uni-stuttgart.de/MoRePaS/software/}.

\begin{figure}[ht]
\centering\includegraphics[width=\textwidth]{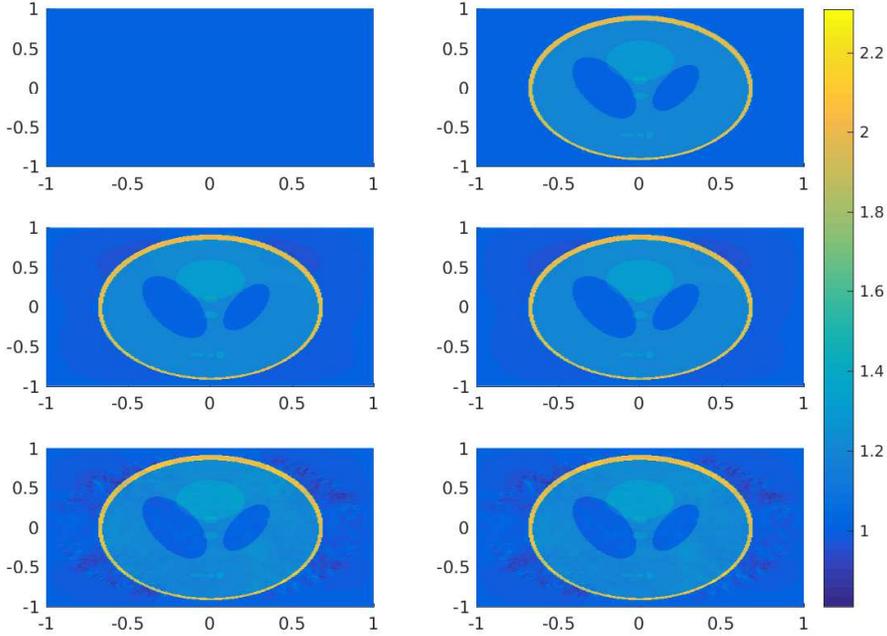}
\caption{From top left to bottom right: $\sigma^0$ the initial guess, $\sigs$ the true conductivity, $\sigma_{BZ}$ the reconstruction via algorithm \ref{algo:BZ_MREITpaper}, $\sigma_{RBZ}$ the reconstruction via algorithm \ref{algo:RBZ_MREITpaper}, $\sigma_{BZ}^\delta$ the reconstruction from noisy data via algorithm \ref{algo:BZ_MREITpaper}, $\sigma_{RBZ}^\delta$ the reconstruction from noisy data via algorithm \ref{algo:RBZ_MREITpaper}.}
\label{fig:compare_reconstructions_MREITpaper}
\end{figure}

As it can be seen in figure \ref{fig:compare_reconstructions_MREITpaper}, all key features of the shepp-logan phantom are captured in the reconstructions using exact data via algorithm \ref{algo:BZ_MREITpaper} (center-left) and algorithm \ref{algo:RBZ_MREITpaper} (center-right) which can visually not be distinguished. This is further reflected via 
\begin{align*}
\frac{\norm{\sigs - \sigma_{BZ}}_{C(\Omega)}}{\norm{\sigma_{BZ}}_{C(\Omega)}} \approx 0.092 
\quad\mbox{and}\quad 
\frac{\norm{\sigma_{RBZ} - \sigma_{BZ}}_{C(\Omega)}}{\norm{\sigma_{BZ}}_{C(\Omega)}} \approx 5.95 \cdot 10^{-4}
\end{align*}
such that a high-resolution image of the conductivity can be obtained using any of the two algorithms. Furthermore, it can be seen that the background is not exactly reconstructed, which strengthens the statement made in remark \ref{rem:conv_proof_MREITpaper}.

Regarding the computational effort of the algorithms, we note that both required the same amount of $14$ iterations (updates of the conductivity) resulting in $28$ solutions of \eqref{eq:forward_problem_weakform_MREITpaper} for algorithm \ref{algo:BZ_MREITpaper}. Algorithm \ref{algo:RBZ_MREITpaper} did update its reduced basis spaces $4$ times resulting in only $8$ solutions of \eqref{eq:forward_problem_weakform_MREITpaper}. The total computational time was $9.84$ seconds for algorithm \ref{algo:BZ_MREITpaper} and $7.61$ seconds for algorithm \ref{algo:RBZ_MREITpaper} resulting in a speed-up of roughly $25\%$. Do note that both algorithms performing $14$ updates of the conductivity have to solve the related PDEs \eqref{eq:Iterationssequenz_MREITpaper} and \eqref{eq:approx_Iterationssequenz_MREITpaper} $14$ times. In our reduced basis approach the PDE \eqref{eq:approx_Iterationssequenz_MREITpaper} remains untouched and one could introduce a third reduce basis space to include this PDE in the adaptive space enrichment procedure as well. Although this should result in further speed-up, the theoretical foundation via theorem \ref{thm:convergence_RBZ_MREITpaper} would then be lost.

Having a look at $\sigma_{BZ}^\delta$ (bottom left of figure \ref{fig:compare_reconstructions_MREITpaper}) and $\sigma_{RBZ}^\delta$ (bottom right of figure \ref{fig:compare_reconstructions_MREITpaper}), the reconstructions via algorithms \ref{algo:BZ_MREITpaper} \& \ref{algo:RBZ_MREITpaper} using the noisy data set, we observe that the key features of the phantom remain intact and note that
\begin{align*}
\frac{\norm{\sigs - \sigma_{BZ}^\delta}_{C(\Omega)}}{\norm{\sigma_{BZ}^\delta}_{C(\Omega)}} \approx 0.13
\quad\mbox{and}\quad 
\frac{\norm{\sigma_{RBZ}^\delta - \sigma_{BZ}^\delta}_{C(\Omega)}}{\norm{\sigma_{BZ}^\delta}_{C(\Omega)}} \approx 9.12 \cdot 10^{-4}.
\end{align*}

The computational effort in this noisy scenario and the speed-up obtained was basically the same as in the noiseless case such that we omit the exact numbers.

\section{Conclusion}
\label{sec:conclusion_MREITpaper}
The $B_z$-based Magnet resonance electrical impedance tomography problem can be solved using the existing Harmonic $B_z$ algorithm. The convergence theory for the algorithm in the two-dimensional setting was extended to include the case when an approximative forward solution of the underlying partial differential equation is used instead of the exact forward solution. This novel result ensures actual numerical convergence of the algorithm and enables the combination of it with innovative numerical methods. The reduced basis method, a model order reduction technique, was presented and a reduced basis version of the \Bz was developed in order to speed-up the algorithm. In a numerical example (including noisy data) a high-resolution image of the shepp-logan phantom was reconstructed. Both algorithms achieved a satisfactory approximation quality and the novel \RBz achieved a speed-up of around $25\%$.

\bibliographystyle{siamplain}
\bibliography{Generalbib}
\end{document}